\documentclass[11pt,dvipsnames]{article}
\usepackage{amsmath}
\usepackage{amsfonts}
\usepackage{amssymb}
\usepackage{amsthm}
\usepackage{mathrsfs}       
\usepackage{graphicx}     
\usepackage{todonotes}
\usepackage[utf8]{inputenc}
\usepackage{xcolor}

\usepackage{amsmath,amsthm}
\usepackage {latexsym}
\usepackage{hyperref}

\usepackage[all]{xy}
\allowdisplaybreaks

\usepackage{xspace}

\newcommand{\R}{\mathbb{R}}
\newcommand{\N}{\mathbb{N}}
\newcommand{\T}{\mathbb{T}}
\newcommand{\Z}{\mathbb{Z}}

\newcommand{\C}{\mathbb{C}}

\newcommand{\vertiii}[1]{{\vert\kern-0.25ex\vert\kern-0.25ex\vert #1 
    \vert\kern-0.25ex\vert\kern-0.25ex\vert}}

\newtheorem{op}{Open Problem}

\newtheorem{thm}{THEOREM}[section]
\newtheorem{remark}[thm]{REMARK}
\newtheorem{lem}[thm]{LEMMA}
\newtheorem{defn}[thm]{DEFINITION}
\newtheorem{prop}[thm]{PROPOSITION}
\newtheorem{cor}[thm]{COROLLARY}

\newcounter{thmbiss}

\def\<{\langle}
\def\>{\rangle}

\numberwithin{equation}{section}

\date{\empty}
\title{Fredholm transformation on Laplacian and  rapid stabilization for the heat equation}
\author{Ludovick Gagnon\thanks{Université de Lorraine, CNRS, Inria équipe SPHINX,  F-54000 Nancy, France. E-mail: \texttt{ludovick.gagnon@inria.fr.}}, \;\; Amaury Hayat\thanks{CERMICS, \'{E}cole des Ponts ParisTech, 6 - 8, Avenue Blaise Pascal, Cité Descartes—Champs sur Marne, 77455 Marne la Vall\'{e}e, France. E-mail: \texttt{amaury.hayat@enpc.fr.}},\;\; Shengquan Xiang\thanks{Bâtiment des Mathématiques, EPFL, Station 8, CH-1015 Lausanne, Switzerland.  E-mail: \texttt{shengquan.xiang@epfl.ch.}}, \;\; Christophe Zhang\thanks{Université de Lorraine, CNRS, Inria équipe SPHINX,  F-54000 Nancy, France. E-mail: \texttt{Christophe.zhang@inria.fr.}}}

\begin{document}
\maketitle

\begin{abstract}
We study the rapid stabilization of the heat equation  on the 1-dimensional torus using the backstepping method with a Fredholm transformation. We prove that, under some assumption on the control operator, two scalar controls are necessary and sufficient to get controllability and rapid stabilization. This classical framework allows us to present the backstepping method with Fredholm transformations on Laplace operators in a sharp functional setting, which is the main objective of this work. Finally, we prove that the same Fredholm transformation also leads to the local rapid stability of the viscous Burgers equation. 
\\

\end{abstract}

\tableofcontents

\section{Introduction}
\label{sec:introduction}

We consider the following heat equation with two internal controls,
\begin{equation}
\left\{ \begin{array}{ll}
\partial_t u - \Delta u = v_{1}(t)\phi_{1}+v_{2}(t)\phi_{2}, & (t,x) \in (0, +\infty) \times \T, \\
u|_{t=0} = u_0(x), & x \in \T,
\end{array} \right.
\label{eq:linearizedWW}
\end{equation}
with $\T=\R/2\pi \Z$ the one-dimensional torus,  $(v_{1},v_{2})\in L^2((0,+\infty);\R)$ real-valued scalar controls to be defined and \textcolor{black}{$(\phi_{1},\phi_{2}) \in H^\sigma(\T;\R)$ with \textcolor{black}{$\sigma \in \R$}} \textcolor{black}{real-valued potentials}. Our goal is to design a suitable feedback law to stabilize the system \eqref{eq:linearizedWW} exponentially with an arbitrarily large decay rate. 

It would be more natural to first consider the simpler system with only one scalar control
\begin{equation}
\left\{ \begin{array}{ll}
\partial_t u - \Delta u = v(t)\phi, & (t,x) \in (0, +\infty) \times \T, \\
u|_{t=0} = u_0(x), & x \in \T,
\end{array} \right.
\label{eq:linearizedWW_single}
\end{equation}
with $v\in L^2((0,+\infty);\R)$ a real-valued control to be defined and $\phi \in H^\sigma(\T;\R)$. As it turns out, as simple as it is, this system is not controllable, due to the degeneracy of the eigenvalues of the Laplacian operator \textcolor{black}{defined on $H^s(\T;\R)$} (see Section \ref{sec:functionalsetting}). In this paper we show that at least two internal controls are required and that, in fact, two controls are enough provided that $\phi_{1}$ and $\phi_{2}$ satisfy some good conditions. In order to obtain the rapid stabilization result we then \textcolor{black}{propose} a double backstepping method, detailed in Section \ref{sec:strategyandoutline}. 

\subsection{Main result}
Let $H^s(\T) , s\in \R$ denotes the classical Sobolev space\footnote{The notation $H^s(\T)$ is shorten to $H^s$ here and there to simplify the notations when the context is clear.} on $\T$ and define $H^{s-}(\T)$ (resp. $(H^{s+}(\T))'$) as $H^{s-\epsilon}(\T)$ (resp. $(H^{s+\epsilon}(\T))'$) for any $\epsilon>0$ small (see \eqref{def-H-}--\eqref{def-H+} for detailed definitions).    Our main result is the following:
\begin{thm}\label{thm-main-linear}
Let  $m\in\mathbb{R}_{+}$ and $\phi_{1},\phi_{2}\in H^{m-1/2-}$, such that 
\begin{equation}
   \phi_{1}=\sum\limits_{n\in\mathbb{N}^{*}}a_{n}^{1}\sin(nx),\;\;    \phi_{2}=\sum\limits_{n\in\mathbb{N}}a_{n}^{2}\cos(nx)
\end{equation}
with,
\[
 a^2_0\neq 0 \; \text{ and } \; c n^{-m}<  |a_n^{k}|< C n^{-m},\;\text{ for }k\in\{1,2\} \text{ and }n\in\N^*.
\]
For any $\lambda>0$, there exist $K_1$ and $K_2$ bounded feedback functionals on 
$H^{m+1/2+}$ such that for any
$y_0\in H^{m+r}$ with $r\in (-1/2, 1/2)$, the equation
\begin{gather}\label{clo-lop-sys-lin}
    \begin{cases}
    \partial_t y-\Delta y= K_{1}(y)\phi_{1} +K_{2}(y)\phi_{2} , & \textcolor{black}{(t,x) \in (0, +\infty) \times \T,} \\
    y(\textcolor{black}{0,\cdot})= y_0, & \textcolor{black}{x\in \T},
    \end{cases}
\end{gather}
has a unique solution satisfying 
\begin{equation}
      y\in C^0([0, +\infty); H^{m+r}\textcolor{black}{(\T)})\cap L^2_{loc}((0, +\infty); H^{m+r+1}\textcolor{black}{(\T)})\cap H^1_{loc}((0, +\infty); H^{m+r-1}\textcolor{black}{(\T)}).
  \end{equation}
  Moreover, we have the following exponential stability estimate,
\textcolor{black}{
  \begin{equation}
      \|y(t,\cdot)\|_{H^{m+r}}\leq C e^{-\lambda t} \|y_{0}\|_{H^{m+r}}, \quad \forall t\in [0, +\infty),
  \end{equation}
  \textcolor{black}{where $C= C_r(\lambda, m)$ is a constant independent of $y_0$.}
  }
\end{thm}
\begin{remark}
Note that the feedback laws $K_{1}$ and $K_{2}$, as well as the backstepping transformations that we construct to obtain them, do not depend on $r\in (-1/2, 1/2)$. Nonetheless, they stabilize the system in the $H^{m+r}$ spaces with $r\in (-1/2, 1/2)$ (for an $H^{m+r}$ initial condition). 
\end{remark}

\begin{remark}
The assumption $a^2_0\neq 0$ is necessary, as it is related to the eigenfunction $1$ of 0 eigenvalue. Otherwise, one can easily check that the ``mass",  $\int_{\T} y(t, x) dx$,  is conserved. In this case, instead of converging to $0$, the solution of the closed-loop system converges exponentially to the final equilibrium state $\tilde{y}(x):= \int_{\T} y_0(x) dx$.
\end{remark}

The same feedback also stabilizes several related  nonlinear systems such  as the viscous Burgers equation and the nonlinear heat equations. 
More precisely, {\color{black} simply as an example}, we have the following theorem  {\color{black} which corresponds to the case $m=r=0$. }
\begin{thm}\label{thm-main-burgers}
Let $\phi_{1},\phi_{2}\in H^{-1/2-}$, such that 
\begin{equation}
   \phi_{1}=\sum\limits_{n\in\mathbb{N}^{*}}a_{n}^{1}\sin(nx),\;\;    \phi_{2}=\sum\limits_{n\in\mathbb{N}}a_{n}^{2}\cos(nx),
\end{equation}
with
\[
 a^2_0\neq 0 \; \text{ and } \;c <  |a_n^{k}|< C ,\;\text{ for }k\in\{1,2\} \text{ and }n\in\N^*.
\]
For any $\lambda>0$, there exists $K_1$ and $K_2$ bounded feedback functionals on 
$H^{1/2+}$ such that, for any
$y_0\in L^{2}$, the equation
\begin{gather}\label{clo-lop-sys-nonlin}
    \begin{cases}
    \partial_t y-\Delta y+\partial_{x}(y^{2}/2)=  K_{1}(y)\phi_{1}+ K_{2}(y)\phi_{2}, \\
    y(0)= y_0,
    \end{cases}
\end{gather}
has a unique solution 
\begin{equation}
      y\in C^0([0, +\infty); L^{2}(\T))\cap L^2_{loc}((0, +\infty); H^{1}(\T))\cap H^1_{loc}((0, +\infty); H^{-1}(\T)).
  \end{equation}
  Moreover, there exists $\delta>0$ such that for any  $\|y_{0}\|_{L^{2}}<\delta$,   \textcolor{black}{we have the following exponential stability estimate}
\textcolor{black}{
  \begin{equation}
      \|y(t,\cdot)\|_{L^2}\leq C e^{-\lambda t} \|y_{0}\|_{L^2}, \quad \forall t\in [0, +\infty),
  \end{equation}
  where $C= C(\lambda)$ is a constant independent of $y_0$.
  }
\end{thm}

\subsection{Related results: the heat equation and the backstepping method}
\label{sec:previouswork}

There exists various way to design feedback laws for infinite dimensional systems: Riccati equations (see for instance \cite{MR1132440, MR3556805} and the references therein), Gramian methods \cite{MR1088227, Urquiza}, Lyapunov functionals (see for instance \cite{bastin2016stability, Coron-trelat-2004, Hayat-2019, Hayat-2019cocv, Xiang-heat-2020}) or  pole-shifting techniques \cite{MR980228,MR472175}, and among others. The backstepping method is also among these methods, and its use for infinite dimensional systems can be traced back to Russell \cite{russell1978controllability} and to Balogh and Krstic \cite{BK1} (we refer to \cite{CoronBook,KKK_Book,Sontag_Book} for an introduction to the finite dimensional backstepping method). Obtained as the continuum limit of a stabilization method applied to space discretizations of PDEs, the backstepping method relied at first on a Volterra transformation of the second kind mapping the system to another stable target system. The Volterra transformation having the advantage of always being invertible, only the existence remains to prove, which is equivalent to solving a PDE of the kernel on a triangular domain. These PDEs usually do not enter in the classical Cauchy problem framework, but different techniques are now known to solve the kernel equation: successive approximations \cite{KrsticSmyshlyaev_Book},  explicit representations \cite{KrsticSmyshlyaev_Book} or method of characteristics \cite{MR4160602}. There exists now a vast literature on the backstepping method with Volterra transformations: let us cite for the heat/parabolic equation \cite{BK1, BK2,2015-Coron-Nguyen-preprint}, for hyperbolic systems \cite{CoronVazquezKrsticBastin13} and for the viscous Burgers equation \cite{MR4265693}. We refer to \cite{KrsticSmyshlyaev_Book} to a general overview of the backstepping method with Volterra transformations.   

As of late, the Fredholm transformation was introduced for the backstepping method as an alternative for certain limitations of the Volterra transformation. In particular, it seems better suited for internal stabilization problems \cite{CGM,zhang:hal-01905098}.
The idea of using transformations remains the same, but proving the existence and invertibility of the transformation is generally more involved. There are mainly two ways to prove the existence of the transformation, either by direct methods \cite{CoronHuOlive16, MR4160602} or, more commonly, by proving the existence of a Riesz basis. For the latter, we again distinguish two cases: either the Riesz basis is deduced directly by an isomorphism applied on an eigenbasis \cite{coron:hal-03161523, Zhang-finite, zhang:hal-01905098} or the existence of a Riesz basis  follows by controllability assumptions and sufficient growth of the eigenvalues of the spatial operator allowing in particular to prove that the family is quadratically close to the eigenfunctions \cite{CGM,CoronLu14,CoronLu15,GLM} (see Section \ref{sec:Rieszdef} and Section \ref{sec:proof-prop} for a definition).

Finally, a large number of papers in the literature have addressed the exponential stability of the heat equations via feedback laws by different means (see for instance \cite{MR3264229, BK1, MR1995484,  MR1878234, MR4159677, MR711000, MR894809, MR2002146} and references therein). We highlight two recent papers on the subject, one on the exponential stability through impulsive feedback for the heat equation \cite{MR3680946}, and the second \cite{Xiang-heat-2020} linked with the quantitative estimate on the stabilization cost $C_r(\lambda,m)$ (we refer to Section \ref{subsubsection:coststab} and Section \ref{subsection:quantitative} on this subject).

\subsection{Contributions}
One of the major contributions of this paper is the thorough study of the optimal functional setting for backstepping transformations for the Laplace operator. 
To illustrate this into perspective, we first present the backstepping problem in a finite-dimensional setting.
\subsubsection{Backstepping in finite dimension}
In a very general setting, finding backstepping transformations for exponential stabilization consists in solving a specific set of equations. Consider the control system, 
\begin{equation}
    \dot{x}= Ax+ Bu,
\end{equation}
the aim is to find suitable a invertible operator $T$ and a feedback law $K$ such that, if $x_1$ is a solution of
\begin{equation}\label{abstract-backsteppin-system}
    \dot{x}_1= Ax_1+ BK x_1,
\end{equation}
then $x_2:= Tx_1$ is a solution of
\begin{equation}\label{abstract-backstepping-target}
    \dot{x}_2= \tilde{A} x_2,
\end{equation}
where $\tilde{A}$ is an exponentially stable operator. In many cases, one makes the classical choice $\tilde{A}=A-\lambda I$.

Formally, we have
\begin{equation}
    \tilde{A} T x_1=\tilde{A}x_2=\dot{x}_2=T \dot{x}_1= T(A+ BK)x_1.
\end{equation}
Thus, $(T,K)$ should solve the following operator equation
\begin{equation}\label{FirstOpEq}
    T(A+BK)=\tilde{A}T.
\end{equation}
This equation is nonlinear in $(T,K)$, which makes it difficult to solve. However, adding the natural condition $TB=B$ to this equation,
\begin{gather}\label{equation:op}
\begin{cases}
    TA+BK= \tilde{A}T, \\
    TB= B,
    \end{cases}
\end{gather}
makes it linear in $(T,K)$ and ensures the existence and uniqueness of a solution, as two controllable systems are always \textit{F-equivalent} (feedback equivalent), which in our case corresponds exactly to equation \eqref{equation:op}.
\begin{thm}[\cite{Brunovsky}]
Let $A, \tilde{A} \in \R^{n\times n}$, $B\in \R^{n\times l}$. If both $(A,B)$ and $(\tilde{A}, B)$ are controllable, then there exists a unique pair $(T,K)\in GL_n(\R)\times \R^{l\times n}$ satisfying equations \eqref{equation:op}.
\end{thm}
The proof of this result relies on the fact that in finite dimension, controllable systems can always be written in control canonical form. Hence, writing \eqref{FirstOpEq} using the canonical form for the pairs $(A,B)$ and $(\tilde{A},B)$ (assuming both are controllable) yields naturally the first equation of \eqref{equation:op} and $TB=B$ (see \cite{coron:hal-03161523} for more details).

We call the system \eqref{equation:op} the \textit{backstepping equations}, and its first equality the \textit{operator equality}. Exponential stabilization is a direct consequence of this equality. Indeed, $x_2$ decays exponentially with rate $\lambda$ and, since $T$ is an invertible operator,  so does $x_1$:
\begin{equation}
\begin{aligned}
    \|x_1(t)\|&=\|T^{-1} x_2(t) \| \\
    &\leq \vertiii{T^{-1}} \|x_2(t)\| \\ 
    &\leq e^{-\lambda t} \vertiii{T^{-1}}  \|x_2(0)\| \\ 
    &\leq e^{-\lambda t}  \vertiii{T^{-1}}\vertiii{T} \|x_1(0)\|.
    \end{aligned}
\end{equation}

Rapid stabilization via backstepping therefore reduces to the existence
of solutions $(T, K)$ to the backstepping equations \eqref{equation:op}.

 In infinite dimension, the situation is more complex, but the same philosophy applies. The role of controllability in finding such backstepping transformations has been established for several important PDE models such as the transport equation \cite{zhang:hal-01905098}, the KdV equation \cite{CoronLu14}, Kuramoto-Sivashinsky equation \cite{CoronLu15}, the linearized Schrödinger equation \cite{CGM}, the linearized Saint-Venant equation \cite{coron:hal-03161523}. We highlight that the uniqueness equation $TB=B$ was decisive in \cite{CGM, coron:hal-03161523, zhang:hal-01905098} to transform non-local terms emanating from distributed control functions into local terms for the operator equality, but was fundamentally used in an implicit way in \cite{CoronLu14, CoronLu15} for boundary controls. 
 
 The general case remains wide open, and is considerably more involved. Indeed, it is not reasonable for instance to expect that two PDEs with wildly differing physical properties should be related to one another by a backstepping transformation. 
 \begin{op}
Let  $A, \tilde{A}$ and $B$ be (unbounded) operators. Is there a necessary and sufficient condition on $A,\tilde{A},B$ to guarantee the existence and the  uniqueness of the solution $(T, K)$ to the operator equality \eqref{equation:op}? 
\end{op}

An important common feature of the infinite-dimensional results mentioned above is that the eigenvalues of the operator $A$ are simple and isolated. This plays a role in the controllability property (through the moments method), which is crucial to implement the backstepping method in the references above. In our case, as we have pointed out above, all non-zero eigenvalues have multiplicity greater than $1$ (this will be specified rigorously in the next section), a feature which occurs in several PDE models, such as the water waves system. As a consequence, the system is not controllable with a single internal control.  This phenomenon appears quite often when working on compact Riemannian manifolds, for example Schrödinger equations on torus. To tackle this difficulty, we consider backstepping with two internal scalar controls.

Backstepping with two scalar controls has already been implemented in \cite{KdVKdV} for a coupled KdV-KdV system with two boundary controls,  for which, because of the coupling nature, it is natural to introduce two control terms.
More recently, it has been improved in \cite{Cerpa-KdVKdV}, where the authors showed that only one control is sufficient to stabilize the system.

 \subsubsection{Sharp functional setting for the Laplacian}

Another important contribution of our paper is to present a sharp functional setting, with respect to the state space and control space, for the application of the backstepping method with a Fredholm transform in the case of the Laplacian with periodic boundary conditions. In particular, we deduce the sharp spaces $H^s(\T)$ for which the Riesz basis exists, which is crucial for the application of the backstepping method with a Fredholm transformation. We hope this precise framework could extend the knowledge on backstepping method using Fredholm type transformation, for example on the use of nonlinear systems,  and on other important models. For instance, Proposition \ref{thm-ope-t}, Corollary \ref{cor-key} and Lemma \ref{lem-rbheat-1}--\ref{lem-rbheat-2} can be similarly proved for Schrödinger equations on $\T$, which somehow extends the analysis of \cite{CGM}. It is  interesting to further investigate whether these new observation could be applied to the bi-linear Schrödinger equations. This analysis could also be applied to evolution equations with fractional Laplacian $(-\Delta)^{\alpha}$, at least for $\alpha$ strictly larger than  $3/4$, where similar results (at least partially, depending on the value of $\alpha$) to Lemma \ref{lem-rbheat-1}  can be proved.\\ 
Indeed, it seems that the critical growth of the eigenvalues for the existence of a quadratically close Riesz basis is $|\lambda_n| \sim n^{3/2+\epsilon}$ (i.e. $\alpha >3/4$ is the case of the fractional Laplacian), meaning that for a growth of order $|\lambda_n| \sim n^{s}, 1\leq s \leq 3/2$ does not seem enough to prove that the family is quadratically close to the eigenbasis. An interesting open problem is therefore to apply the backstepping method with a Fredholm type transformation for spatial operator with eigenvalues with growth $|\lambda_n| \sim n^{s}, 1< s \leq 3/2$, the first-order equation being excluded due to the positive answer \cite{coron:hal-03161523, zhang:hal-01905098} and its link to other transformation such as  the Hilbert transform. However the case $1< s \leq 3/2$ remains and this analysis may give some inspiration on the study of fractional Laplacian with the value of $\alpha$ lower than this threshold.

Moreover, we  also highlight that the framework investigated here is closely related to the one found for the linear Schr\"odinger equation in \cite{CGM}. A major distinction between \cite{CGM} and the present article is that the well-posedness for the closed-loop system here relies on the dissipation properties of the heat equation. Hence, there is no need to satisfy the operator equality $T(A+BK)=(A-\lambda I)T$ for functions in $D(A+BK)$ (see Remark  \ref{rmk:regularityoperatorequality}).  \\

\subsubsection{Cost of stabilization}\label{subsubsection:coststab}

Finally, as we can see from Theorem \ref{thm-main-linear} (more precisely, from Section \ref{sec:proof-thms}), it is totally a new observation that the same feedback law stabilizes the system in $H^{m+r}$ with some cost $C_r(\lambda, m)$ depending on $r\in (-1/2, 1/2)$ and $\lambda\notin \mathcal{N}$. It is an important but challenging problem to get any quantitative description of this constant $C_r(\lambda, m)$ in view of possible applications. Such a description would also be interesting for finite time stabilization problems, namely $Ce^{C\lambda^{\beta}}$ type estimates. So far such estimates have been achieved via different methods, relying on direct energy estimates \cite{2015-Coron-Nguyen-preprint}, Bessel functions \cite{MR3916517}, iterative methods \cite{XiangKdVNull}, and spectral inequalities \cite{Xiang-heat-2020, Xiang-NS-2020}.  However, so far such important property  has not yet been discovered for Fredholm type backstepping methods. Armed with the precise description introduced in this paper, we believe that we are closer to an answer.

\subsection{Structure of the paper} 
This paper is organized as follows: in Section \ref{sec:functionalsetting} we define the functional setting, in Section \ref{sec:noncontrol} we show that the system \eqref{eq:linearizedWW_single} with a single control is neither controllable nor stabilizable, in Section \ref{sec:strategyandoutline} we present the double backstepping approach. In Sections \ref{sec:proof-prop} and \ref{sec:well-pos} we prove the main propositions and lemma on which the backstepping method relies. Finally in Section \ref{sec:proof-thms} we prove Theorem \ref{thm-main-linear}--\ref{thm-main-burgers}.\\

\section{Functional setting}
\label{sec:functionalsetting}
\subsection{Function spaces}
We start by recalling some results on the eigenvectors and eigenvalues of the Laplacian on the torus. Observe that the classical Fourier series $\{e^{in x}\}_{n\in \Z}$ in $\T$ are eigenfunctions of the Laplacian operator $\Delta$ associated to the eigenvalues
$\lambda_{n}:=-n^{2}$, and form an orthonormal basis of $L^{2}(\mathbb{T})$. Note that thanks to the fact that we are working on $\mathbb{T}$ without boundary, the Sobolev space $H^s$ coincides with the span of $\{n^s e^{inx}\}$. Note also that 
\begin{equation*}
    \lambda_{n}=\lambda_{-n}=-n^{2},
\end{equation*}
and therefore the nonzero eigenvalues have multiplicity 2. We therefore easily deduce the existence of an orthonormal basis of real-valued eigenfunctions for the Laplacian: 
\begin{equation}
    \begin{split}
    f_n^{1}:= \frac1{\sqrt{\pi}}\sin nx,\; \; f_n^{2}:= \frac1{\sqrt{\pi}} \cos nx, & \textrm{ associated to } \lambda_n:= -n^2, \; \forall n\in \N^*,\label{def:eigenvec} \\
    {\color{black}f^2_0}:= \frac1{\sqrt{2\pi}}, & \textrm{associated to}  \lambda_0:= 0.
    \end{split}
\end{equation}
Except in Section \ref{sec:noncontrol} where using the Fourier series $\{e^{in x}\}_{k\in \Z}$ is convenient, we will use this basis, and work with real-valued functions, i.e. $L^2(0,2\pi; \R)$.
We further define 
\begin{gather}
    L^2_1:= \textrm{span}_\R \{\sin nx\}_{n\in \N^*}, \textrm{ describing the odd functions}, \\ 
    L^2_2:= \textrm{span}_\R \{\cos nx\}_{n\in \N},  \textrm{ describing the even functions}, \\
    \textrm{$L^2_i$ is a subspace of $L^2$ that is endowed with the same norm,} \\
    L^2(\T)= L_1^2\oplus L_2^2.
\end{gather}
Similarly, concerning Sobolev space s in $\T$, recall that one has, {\color{black}\begin{equation}
     H^{m}\textcolor{black}{(\T)} \textcolor{black}{=} \left\{f =\sum\limits_{n\in\mathbb{N}^{*}} a^1_{n} f_{n}^{1}+ \sum\limits_{n\in\mathbb{N}} a^2_{n} f_{n}^{2} \; \; \right| \left. \; a_n^i\in \R \textrm{ and } \sum\limits_{n\in\mathbb{N}^{*}} n^{2m}\left((a^1_{n})^2+ (a^2_n)^2\right)< +\infty
    \right\},
\end{equation}
with the (inhomogeneous) Sobolev norm
\begin{equation}
    \|f\|_{H^m}^2:= (a^2_0)^2+ \sum\limits_{n\in\mathbb{N}^{*}} n^{2m}\left((a^1_{n})^2+ (a^2_n)^2\right).
\end{equation}

In this paper \textcolor{black}{we say that the function $y$ belongs} to $H^{m-}$, and \textcolor{black}{that the functional} $\mathcal{L}: H^{m+}\rightarrow \R$ \textcolor{black}{is} bounded, \textcolor{black}{when}
\begin{gather}
    y\in H^{m-\varepsilon} \; \textrm{ for all } \varepsilon>0, \label{def-H-}\\
    \mathcal{L}: H^{m+\varepsilon}\rightarrow \R \textrm{ is bounded for all } \varepsilon>0. \label{def-H+}
\end{gather}
}

We also define the sub-spaces $H^{m}_{i}$, $i\in\{1,2\}$ as follows, for $m\in\mathbb{R}$
\begin{gather}
    H^{m}_{1} :=\{a\in H^{m}(\T)\;|\;a =\sum\limits_{n\in\mathbb{N}^{*}} a_{n} f_{n}^{1}
    \}, \\
    H^{m}_{2} :=\{a\in H^{m}(\T)\;|\;a =\sum\limits_{n\in\mathbb{N}} a_{n} f_{n}^{2}
    \}, \\
     H^m\textcolor{black}{(\T)}= H_1^m\oplus H_2^m.
\end{gather}
{\color{black} Notice that 
\begin{equation}
    \Delta: H^{m}_k\rightarrow H^{m-2}_k.
\end{equation}
We remark here that for $f\in H^{m+1}_1, g\in H^{m-1}_1$, the inner product $\langle\cdot, \cdot \rangle_{H^m_1}$ is well-defined and is given by 
\begin{equation*}
      \langle f, g\rangle_{H^m_1}= \sum_{n\in \N^*} (n^m f_n) (n^m g_n)= \sum_{n\in \N^*} (n^{m+1} f_n) (n^{m-1} g_n),
\end{equation*}
which, inspired by the last term of the preceding formula, can be also denoted as $\langle\cdot, \cdot \rangle_{H^{m+1}_1, H^{m-1}_1}$.

In order to describe the precise definition domain of the operator $T$, \textcolor{black}{we recall that for} the Schwartz type space $\mathcal{S}\textcolor{black}{(\T)}$ satisfying fast decay at high frequency, \textcolor{black}{one has} \begin{gather}
     \mathcal{S}(\T) =\bigg\{f =\sum\limits_{n\in\mathbb{N}^{*}} a^1_{n} f_{n}^{1}+ \sum\limits_{n\in\mathbb{N}} a^2_{n} f_{n}^{2} \;\;  \bigg| \; \; \forall m\in \N, \forall \varepsilon>0, \; \exists M \textrm{ such that } \\  \qquad \qquad \qquad \qquad   n^{2m}\left((a^1_{n})^2+ (a^2_n)^2\right)< \varepsilon, \; \; \forall n> M
    \bigg\}.
\end{gather} }
We also  define the decomposition of in odd and even functions as follows
\begin{gather} 
    \mathcal{S}_{1}:= \big\{a\in \mathcal{S} \; \big|  \; \langle a,f_{n}^{2}\rangle =0,\; \forall n\in\mathbb{N}\big\},\\
     \mathcal{S}_{2}:= \big\{a\in \mathcal{S} \; | \; \langle a,f_{n}^{1}\rangle =0,\; \forall n\in\mathbb{N}^{*}\big\}.
\end{gather}
{\color{black} While, by denoting  the space $ \mathcal{S}'$ as the dual of $ \mathcal{S}$, we also define 
\begin{gather} 
    \mathcal{S}_{1}':= \big\{a\in \mathcal{S}' \; | \; \langle a,f_{n}^{2}\rangle =0,\; \forall n\in\mathbb{N}\big\},\\
     \mathcal{S}_{2}':= \big\{a\in \mathcal{S}' \; | \; \langle a,f_{n}^{1}\rangle =0,\; \forall n\in\mathbb{N}^{*}\big\},
\end{gather}
which strictly speaking is not the dual of $\mathcal{S}_k$, but the quotient space  $\mathcal{S}'/\mathcal{S}_{k+1}$.
We  easy observe that, 
\begin{align}
    \mathcal{S} & \subset H^s\subset  \mathcal{S}', \; \forall s\in \R, \\
    \mathcal{S}_k & \subset H^s_k\subset  \mathcal{S}_k', \; \forall s\in \R, \; \forall k\in \{1, 2\}.
\end{align}
}

\subsection{Riesz bases}\label{sec:Rieszdef}
Finally we recall here the definition of a Riesz basis, see for instance the book \cite{christensen2003introduction}, or a monograph on the moment theory \cite{moment-book} (as well as the references cited in \cite{CGM}).
{\color{black}
\begin{defn}[Vector family]\label{def-riesz-bas}
Let $X$ be a Hilbert space. A family of vectors $\{\xi_n\}_{n\in \textcolor{black}{\mathcal{I}}}$, \textcolor{black}{where $\mathcal{I}=\Z$, $\N$, or $\N^*$} is said to be
\begin{itemize}
    \item[(1)] \textbf{Minimal} in X, if for every $k\in \textcolor{black}{\mathcal{I}}$, $\xi_k\notin \overline{span\{\xi_i; i\in \textcolor{black}{\mathcal{I}}-\{k\}\}}$.  
    \item[(2)]  \textbf{Dense} in X, if  $ \overline{span\{\xi_i; i\in \textcolor{black}{\mathcal{I}}\}}= X$. 
     \item[(3)]  \textbf{$\omega$-independent} in X, if  
     \begin{equation}
       \sum_{k\in \textcolor{black}{\mathcal{I}}} c_k \xi_k=0 \textrm{ in $X$ with } \{c_n\}_{n\in\textcolor{black}{\mathcal{I}}}\in\ell^2(\mathcal{I}) \Longrightarrow   c_n=0,\, \forall n\in \textcolor{black}{\mathcal{I}}.
            \end{equation}
     \item[(4)]  \textbf{Quadratically close} to a family of vector $\{e_n\}_{n\in \textcolor{black}{\mathcal{I}}}$, if
     \begin{equation}
         \sum_{k\in \textcolor{black}{\mathcal{I}}} \|\xi_k- e_k\|_X^2<+\infty.
     \end{equation}
      \item[(5)]  \textbf{Riesz basis} of $X$, if it is the image of an isomorphism (on $X$) of some orthonormal basis.
       \item[(5)']  \textbf{Riesz basis} of $X$ (an equivalent definition of (5)), if it is  dense in $X$ and if there exist $C_1, C_2>0$ such that for any $\{a_n\}_{n\in \textcolor{black}{\mathcal{I}}}\in\ell^2(\mathcal{I})$ we have 
       \begin{equation}\label{ineq-riesz}
           C_1\sum_{k\in \textcolor{black}{\mathcal{I}}} |a_k|^2\leq  \|\sum_{k\in \textcolor{black}{\mathcal{I}}} a_k \xi_k\|_{X}^2\leq  C_2\sum_{k\in \textcolor{black}{\mathcal{I}}} |a_k|^2.
       \end{equation}
\end{itemize}
\end{defn}
\textcolor{black}{These definitions allow us to give the} following criteria for a Riesz basis \textcolor{black}{which} will be used later on in Section \ref{sec-riesz-pro}.
\begin{lem}\label{lem-cri-riesz}
Let $\{\xi_n\}_{n\in\textcolor{black}{\mathcal{I}}}$ be quadratically close to an orthonormal basis $\{e_n\}_{n\in\textcolor{black}{\mathcal{I}}}$. Suppose that  $\{\xi_n\}_{n\in\textcolor{black}{\mathcal{I}}}$ is either dense in $X$ or $\omega$-independent in $X$, then $\{\xi_n\}_{n\in\textcolor{black}{\mathcal{I}}}$ is a Riesz basis of $X$.
\end{lem}

 It is noteworthy that under the assumption that $\{\xi_n\}_{n\in \mathcal{I}}$ is quadratically close to some orthonormal basis $\{e_n\}_{n\in \mathcal{I}}$,  if for some coefficients $\{c_n\}_{n\in \mathcal{I}}\in \R^{\mathcal{I}}$ (or $\C^{\mathcal{I}}$)  the following series converges in the Cauchy sense 
 \begin{equation}
     \sum_{k\in \mathcal{I}} c_k \xi_k  \textrm{ converges in } X ,
 \end{equation} which in particular contains the case that it converges to 0, then automatically we know that  the coefficients $\{c_n\}_{n\in \mathcal{I}}$ belong to  $l^2(\mathcal{I})$.  Indeed, we know  from the fact that the series converges in $X$ that, for every $N$,
 \begin{equation}
     \sum_{k>|N|, k\in \mathcal{I}} c_k \xi_k  \textrm{ converges in } X,
 \end{equation}
 which means that the $X$ norm of the preceding series is finite:
 \begin{equation}
      \left\|\sum_{k>|N|, k\in \mathcal{I}} c_k \xi_k\right\|_X<+\infty.
 \end{equation}
 
 Since  $\{\xi_n\}_{n\in \mathcal{I}}$ is quadratically close to $\{e_n\}_{n\in \mathcal{I}}$, for some $N$ sufficiently large we have
 \begin{equation}
     \sum_{k>|N|, k\in \mathcal{I}} \|(\xi_k- e_k)\|_X^2<\frac{1}{4},
 \end{equation}
 which, to be combined with the fact that $\{e_n\}_{n\in \mathcal{I}}$ is an orthonormal basis of $X$, yield
 \begin{align*}
     \left\|\sum_{k>|N|, k\in \mathcal{I}} c_k \xi_k\right\|_X&= \left\|\sum_{k>|N|, k\in \mathcal{I}} c_k e_k+ \sum_{k>|N|, k\in \mathcal{I}} c_k (\xi_k- e_k)\right\|_X, \\
     &\geq \left\|\sum_{k>|N|, k\in \mathcal{I}} c_k e_k\right\|_X- \left\| \sum_{k>|N|, k\in \mathcal{I}} c_k (\xi_k- e_k)\right\|_X\\
     & \geq \left(\sum_{k>|N|, k\in \mathcal{I}} (c_k)^2\right)^{1/2}- \left(\frac{1}{4}\sum_{k>|N|, k\in \mathcal{I}} (c_k)^2\right)^{1/2}\\
      & =\frac{1}{2} \left(\sum_{k>|N|, k\in \mathcal{I}} (c_k)^2\right)^{1/2}.
 \end{align*}
 Therefore, $\{c_n\}_{n\in \mathcal{I}}\in\ell^2(\mathcal{I})$.

We also have the following Lemma that will be useful in the Section \ref{sec-riesz-pro}.
\begin{lem}\label{lem-iso-riesz-equi}
Let $X$, $Y$ be Hilbert spaces. Let $T:X\rightarrow Y$ be an isomorphism. Suppose that $\{\xi_n\}_{n\in \textcolor{black}{\mathcal{I}}}$ is a Riesz basis of $X$, then with $\zeta_n:= T\xi_n$, \textcolor{black}{the family} $\{\zeta_n\}_{n\in \textcolor{black}{\mathcal{I}}}$ is a Riesz basis of $Y$.
\end{lem}
\begin{proof}
 If $X=Y$ we can directly use Definition \ref{def-riesz-bas} (5) to show that it is a Riesz basis \textcolor{black}{as the image of an orthonormal basis by an isomorphism on $X$}. Otherwise,  a quite straightforward proof is given according to Definition \ref{def-riesz-bas} (5)'. 
 
 We first show the inequality \textcolor{black}{\eqref{ineq-riesz}}. Thanks to the fact that $T$ is an isomorphism and that $\{\xi_n\}_{\Z}$ is a Riesz basis of $X$, \textcolor{black}{there exist constants $C$, $C_{1}$ and $C_{2}$ such that}
 \begin{equation*}
    \left\|\sum_{k\in \textcolor{black}{\mathcal{I}}} a_k \zeta_k\right\|_{Y}^2=  \left\|T \sum_{k\in \textcolor{black}{\mathcal{I}}} a_k \xi_k\right\|_{Y}^2\leq C\left\| \sum_{k\in \textcolor{black}{\mathcal{I}}} a_k \xi_k\right\|_{X}^2\leq  C C_2\sum_{k\in \textcolor{black}{\mathcal{I}}} |a_k|^2,
 \end{equation*}
 and
 \begin{equation*}
    \left\|\sum_{k\in \textcolor{black}{\mathcal{I}}} a_k \zeta_k\right\|_{Y}^2=  \left\|T \sum_{k\in \textcolor{black}{\mathcal{I}}} a_k \xi_k\right\|_{Y}^2\geq 
    \textcolor{black}{C^{-1}}\left\| \sum_{k\in \textcolor{black}{\mathcal{I}}} a_k \xi_k\right\|_{X}^2\geq  \textcolor{black}{C^{-1}} \textcolor{black}{C_{1}}\sum_{k\in \textcolor{black}{\mathcal{I}}} |a_k|^2.
 \end{equation*}
 Next, we show that $\{\zeta_n\}_{\textcolor{black}{\mathcal{I}}}$ is dense in $Y$. For any $\zeta\in Y$, since $T^{-1}\zeta=: \xi\in X$, for any $\varepsilon>0$ there exist a finite combination such that 
 \begin{equation}
     \left\|\sum_{|k|\leq N} a_k \xi_k- \xi\right\|_X< \varepsilon.
 \end{equation}
 Thus 
 \begin{align*}
     \left\|\zeta- \sum_{|k|\leq N} a_k \zeta_k \right\|_Y&=  \left\|T\xi- T\sum_{|k|\leq N} a_k \xi_k \right\|_Y \\
     &\leq C   \left\|\zeta- \sum_{|k|\leq N} a_k \xi_k \right\|_X \\
     &\leq C \varepsilon,
 \end{align*}
 which concludes the proof of the lemma.
\end{proof}

\begin{lem}[Proposition 19 of \cite{Beauchard-Laurent}]
Let $X$ be a Hilbert space. Suppose that $\{\xi_n\}_{n\in \mathcal{I}}$ is a Riesz basis of $X$, then its bi-orthogonal sequence $\{\xi_n'\}_{n\in \mathcal{I}}$ is also  a Riesz basis of $X$, where  bi-orthogonal means,
\begin{equation}
    \langle \xi_n, \xi_m'\rangle_{X}= \delta_{n, m}, \; \forall n, m\in \mathcal{I}.
\end{equation}
For any $f\in X$, there exists a unique sequence $\{a_k\}_{k\in \mathcal{I}}\in\ell^2(\mathcal{I})$ such that 
\begin{equation}
    f= \sum_{k\in \mathcal{I}} a_k \xi_k \textrm{ in } X,
\end{equation}
where the series converges in $X$ under Cauchy sequence sense.  Moreover, 
\begin{gather}
    a_k:=  \langle f, \xi_k'\rangle_{X},\\
   C_1 \sum_{k\in \mathcal{I}} |a_k|^2\leq  \|f\|_X^2\leq C_2   \sum_{k\in \mathcal{I}} |a_k|^2.
\end{gather}
\end{lem}
}

\section{From one to two controls}
\label{sec:noncontrol}
As we have mentioned in the introduction, due to the multiplicity of its nonzero eigenvalues, our system is not controllable with a single internal scalar control. We prove this assertion, and prove that controllability can however be achieved with two internal scalar controls. We then lay out how we adapt the backstepping method to obtain a stabilization result with two controls.

\subsection{Non-controllability with one scalar control}
\textcolor{black}{In this section we show that, for any $T>0$,  the system \eqref{eq:linearizedWW_single} is not controllable. Recall that $\{e^{inx}\}_{n\in\mathbb{Z}}$ is a basis of eigenfunctions of $A$ associated to the eigenvalue $\lambda_{n}=-n^{2}$.} {\color{black} For ease of the presentation \textcolor{black}{and for} symmetry \textcolor{black}{considerations}, in this section we choose to work with the orthonormal basis $\{e^{inx}\}_{n\in\mathbb{Z}}$ instead of $\{\sin nx, \cos nx\}$.} \textcolor{black}{Therefore, either
\begin{equation}\label{eq:phi-non-zero}
    \langle \phi, e^{inx}\rangle \neq 0,\;\forall\; n\in \mathbb{Z},
\end{equation}
or the system is non-controllable as there exists $n_{0}$ such that $\langle \phi, e^{in_{0}x}\rangle = 0$ and therefore the control has no effect on the associated one dimensional vector space: $e^{-n_{0}^{2}t+in_{0}x}$ is a solution of \eqref{eq:linearizedWW_single} with initial condition $e^{in_{0}x}$, whatever the control is. Assume now that \eqref{eq:phi-non-zero} holds.} Motivated by \textcolor{black}{the moment method (see for instance \cite{Russell-F})} and the fact that $\lambda_n=\lambda_{-n}$ we have \textcolor{black}{for a solution $u$ to \eqref{eq:linearizedWW_single} with $u(0)=0$,}
\begin{equation}
    \frac{\langle u(\textcolor{black}{T,\cdot}), e^{inx\rangle}}{\langle \phi, e^{inx}\rangle}=  \frac{\langle u(\textcolor{black}{T,\cdot}), e^{-inx\rangle}}{\langle \phi, e^{-inx}\rangle},
\end{equation}
\textcolor{black}{
thus let us denote
\begin{equation}
    d_{n}:= \frac{\langle \phi, e^{inx}\rangle}{\langle \phi, e^{-inx}\rangle},
\end{equation}
}
one has
\begin{equation}
\textcolor{black}{\langle u(\textcolor{black}{T,\cdot}), e^{inx}\rangle=  d_{n}\langle u(\textcolor{black}{T,\cdot}), e^{-inx}\rangle.}
\end{equation}
Hence the \textcolor{black}{states} $u(\textcolor{black}{T,\cdot})$ that are reachable \textcolor{black}{at time $T$} satisfy the following 
\begin{equation}\label{notcon}
    u(\textcolor{black}{T,x})= d_0+  \sum_{n\geq 1} k_n \left(d_n e^{inx}+ e^{-inx}\right),
\end{equation}
which means that the projection of the reachable space on the two dimensional space 
\textcolor{black}{$\text{Span}\{e^{inx}, e^{-inx}\}$}
is always of one dimension, hence the system is not controllable.
\\  

\textcolor{black}{This non-controllability prevents any stabilization result.} For instance, we can simply consider the space $\textcolor{black}{\text{Span}}\{\sin nx, \cos nx\}$, as we are only allowed to change the direction of $\left(a_n e^{inx}+ e^{-inx}\right)$,  the projection of the solution on its co-direction $\left( e^{inx}-a_n e^{-inx}\right)$ does not change, thus the solution is not asymptotically stable whatever the feedback control.

The setting is different from the one found in \cite{Lebeau-Robbiano-CPDE} for controllability  and \cite{Xiang-heat-2020} for finite time stabilization (the special case for $\T$, as these papers deal with general compact Riemannian manifolds),  where the controllability and the stabilizability is obtained under the assumption that the controlled domain is $\omega\subset \textcolor{black}{\mathbb{T}}$, for which the control has infinite degrees of freedom. \textcolor{black}{and not a one dimensional scalar control}. \\

\subsection{Controllability with two scalar controls}

{\color{black}According to the preceding section, two controls are required for the controllability of the heat equation on $\T$:  $\phi_1 v_1(t)+ \phi_2 v_2(t)$, which corresponds to the system \ref{eq:linearizedWW}.
\textcolor{black}{In fact}, two controls are eventually sufficient.  In the following we prove the controllability in $L^2(\T)$ space, while the other spaces can be treated similarly.

\textcolor{black}{D}ue to the smoothing effect of the heat equation we only consider the so-called null controllability, $i. e.$  for any $u_0\in L^2(\T)$ there exist $v_1, v_2\in L^2(0, T)$ (this space is the natural space according to Lions' \textcolor{black}{Hilbert Uniqueness Method}, though this is not the optimal candidate) such that the final state becomes $0$. In order to simplify the presentation, we always assume the projections of $\phi_1, \phi_2, u_0$ on the \textcolor{black}{direction corresponding to the eigenfunction} $e^{i0x}$ to be 0.
Assuming that $$u_0= \sum_{n\in \N^*} b_n^1e^{inx}+ b_n^2e^{-inx}\in L^2(\T),$$
Direct calculation yields,
\begin{align*}
u(T)= & \int_0^{T} e^{A (T-s)}(\phi_1 v_1(t)+ \phi_2 v_2(t)) ds+ \int_0^{T} e^{A(T-s)} u_0 ds  \\
= &  \sum_{n\in \textcolor{black}{\N^*}} \left(\int_0^{T} e^{\lambda_n (T-s)}\left(\<\phi_1,e^{inx}\> v_1(s)+ \<\phi_2,e^{inx}\> v_2(s)\right) ds \right) e^{inx}\\
\textcolor{black}{+}&\sum_{n\in \textcolor{black}{\N}^*} \left(\int_0^{T} e^{\lambda_n (T-s)}\left(\<\phi_1,e^{-inx}\> v_1(s)+ \<\phi_2,e^{-inx}\> v_2(s)\right) ds \right) e^{-inx}, \\
& \;\; + \sum_{n\in \textcolor{black}{\N}^*} \left(\frac{b^1_n}{e^{n^2 T}} e^{inx}+ \frac{b^2_n}{e^{n^2 T}} e^{-inx}\right). 
\end{align*}

The preceding  formula indicates that \textcolor{black}{the null controllability requires} that 
\begin{equation*}
\left(\int_0^T e^{\lambda_n (T-s)} v_1(s) ds , \int_0^T e^{\lambda_n (T-s)} v_2(s) ds \right)
\begin{pmatrix}
\<\phi_1,e^{inx}\> & \<\phi_1,e^{-inx}\> \\
\<\phi_2,e^{inx}\> & \<\phi_2,e^{-inx}\>
\end{pmatrix}
= \textcolor{black}{-}\left(\frac{b_n^1}{e^{n^2 T}}, \frac{b_n^2}{e^{n^2 T}}\right).
\end{equation*}
Consequencely,  for any $n\in \N^*$ provided that the following matrix is invertible, we can control the two dimensional space \textcolor{black}{S}pan$\{e^{inx}, e^{-inx}\}=$ \textcolor{black}{S}pan$\{\cos nx, \sin nx\}$,
\begin{equation*}
\begin{pmatrix}
\<\phi_1,e^{inx}\> & \<\phi_1,e^{-inx}\> \\
\<\phi_2,e^{inx}\> & \<\phi_2,e^{-inx}\>
\end{pmatrix}, \;\; \forall n\in \N^*.
\end{equation*}

It is a quite general assumption to achieve, a simple example can be
\begin{align*}
    \phi_1(x)= \sum_{n\in\N^*} c_n^1 e^{inx}, \; 
     \phi_2(x)= \sum_{n\in \N} c_n^2 e^{-inx},
\end{align*}
with $c^1_n c^2_n\neq 0$ \textcolor{black}{for any $n\in\mathbb{N}^{*}$. Note that, as mentioned earlier, we excluded the direction $1=e^{i0x}$ corresponding to the case $n=0$ to simplify the presentation but it could be included as well}.
Heuristically, this case already provides the exact controllability in projections on finite dimensional subspaces, for example,  span$\{e^{inx}: -N< n< N\}$. 
Moreover, if further $c^1_n, c^2_n$ satisfy some suitable growth assumption, the system is even exact null controllable in $L^2(\T)$.

\begin{prop}\label{prop-con}
If there exist $-\infty< \alpha\leq \beta< 1/2$ and $c, C>0$ such that
\begin{equation}\label{eq:cond-controllable}
    c n^{\alpha}\leq |c_n^1|, |c_n^2|\leq  C n^{\beta}, \forall n\in \N^*
\end{equation}
then the system \eqref{eq:linearizedWW} is $L^2(\T)$ exact null controllable.
\end{prop}
\begin{remark}
The assumption on $\beta<1/2$ is \textcolor{black}{here} to guarantee that the functions $\phi_i\in H^{-1}$\textcolor{black}{.} Thus for any given $v(t)\in L^2(0, T)$ the inhomogeneous term $v(t)\phi_i(x)$ belongs to $L^2((0, T); H^{-1})$, which indicates that  the open-loop system is well-posed in $C^0([0, T); L^2)\cap L^2((0, T); H^1)$. \textcolor{black}{On the other hand,} the assumption $-\infty< \alpha$ is \textcolor{black}{used for} the null controllability property that will be proved in the following. \textcolor{black}{T}he lower bound on $c_n^i$ proposed here is not the sharp condition. As we can see from  the following proof, to get the null controllability in $L^2(\T)$ space, it suffices to find some $0< T_0< T$ and $C>0$ such that
\begin{equation}
    C e^{-T_0 n^2} \leq |c_n^1|, |c_n^2|, \forall n\in \N^*,
\end{equation}
which is of course weaker than the condition proposed in Proposition \ref{prop-con}.
\end{remark}
\begin{proof}[Proof of Proposition \ref{prop-con}] 
In order to solve the controllability problem, it suffices to treat the following moment problem: show that for any $\{b_n^1\}_{\N^*}, \{b_n^2\}_{\N^*}\in \ell^2(\N^*)$, there exist $v_1(t), v_2(t)\in L^2(0, T)$ such that 
\begin{align}
   \int_0^{T} e^{-n^2 (T-s)} v_1(s) ds&= \frac{b_n^1}{e^{n^2 T} c^1_n}, \;\; \forall n\in \N^*, \label{mon-v1-lin} \\
 \int_0^{T} e^{-n^2 (T-s)} v_2(s) ds&= \frac{b_n^2}{e^{n^2 T} c^2_n}, \;\; \forall n\in \N^*. \label{mon-v2-lin}
\end{align}
\textcolor{black}{Solving this problem is, in fact,} a direct consequence of the following moment theory. 
\begin{lem}[\cite{Russell-F}, Section 3, Equation (3.25)]\label{lem-mom-rf}
For any $T>0$. The sequence  $\{e^{-n^2(T-s)|_{s\in (0, T)}}\}_{n\in \N^*}$ is minimal in $L^2(0, T)$, thus \textcolor{black}{admits} a bi-orthogonal sequence $\{\Psi_n\}_{n\in \N^*}$ satisfying
\begin{equation}
    \int_0^{T} e^{-n^2(T-s)} \Psi_m(s) ds= \delta_{n, m}, \forall n,m \in \N^*.
    \end{equation}
Moreover, there exists $C>0$ such that
\begin{equation}
        \|\Psi_n\|_{L^2(0, T)}\leq C e^{C n}.
\end{equation}
\end{lem}
By adapting Lemma \ref{lem-mom-rf}, \textcolor{black}{and assuming \eqref{eq:cond-controllable},} we know that the moment problem \eqref{mon-v1-lin} can be solved by \textcolor{black}{setting}
\begin{equation}
    v_1(s):= \sum_{n\in N^*} \frac{b_n^1}{e^{n^2 T} c^1_n} \Psi_n,
\end{equation}
satisfying, \textcolor{black}{from \eqref{eq:cond-controllable},}
\begin{equation}
    \|v_1\|_{L^2(0, T)}\leq \sum_{n\in N^*}  C b^1_n n^{-\alpha} e^{C n- Tn^2}\leq  C_{2}\sum_{n\in N^*}   b^1_n  e^{-Tn^2/2}<+\infty,
\end{equation}
\textcolor{black}{where $C_{2}>0$ is a constant independent of $n$.} Then \textcolor{black}{a} similar procedure leads to $v_2(t)$ as the solution of the moment problem \eqref{mon-v2-lin}.
\end{proof}

}

\subsection{Double backstepping: strategy and outline}
\label{sec:strategyandoutline}

Inspired by the fact that \textcolor{black}{$(\sin(nx),\cos(nx))$ form an orthonormal basis of the two-dimensional eigenspaces corresponding to the eigenvalue $\lambda_{n}$,} we directly consider the special form of $\phi_k$:
\begin{gather}
    \phi_1:= \sum_{n\in \N^*} a_n^{1} \sin nx\in L^2_1, \\
     \phi_2:= \sum_{n\in \N} a_n^{2} \cos nx\in L^2_2.
\end{gather}
We can \textcolor{black}{similarly} separate the function $y(t)$ in
\begin{equation}
    \textrm{$y(t)=: y_1(t)+ y_2(t)$ with $y_k\in L_k^2$.}
\end{equation}
Therefore,
\begin{equation}\label{eqk}
\left\{
\begin{aligned}
    \partial_t y_1- \partial_x^2 y_1&= u_1(t)\phi_1 , \\
    \partial_t y_2- \partial_x^2 y_2&=u_2(t) \phi_2 
    \end{aligned}
    \right.
\end{equation}
\textcolor{black}{The logic behind is to deal with the odd functions using the first control and with the even functions using the other one. What we are going to show is that each of the systems \eqref{eqk} can be rapidly stabilized for $k\in\{1,2\}$.}\\

\textcolor{black}{To do so, we would like to show that for any $k\in\{1,2\}$ and for any $\lambda>0$, under some conditions on $\phi_{k}$,}
there exists \textcolor{black}{an isomorphism} $T_k(\lambda): L^2_k\rightarrow L^2_k$ as well as a feedback $u_k(t):= K_k(\lambda) y_k(t, x)$ such that the solution of \begin{gather*}
\begin{cases}
    \partial_t y_k- \partial_x^2 y_k= \phi_k K_k y_k, \\
    y_k(0)\in L^2_k
\end{cases}    
\end{gather*}
is such that $z_k:= T_k(\lambda) y_k$ satisfies the following equation
\begin{equation}\label{eq:target}
    \partial_t z_k- \partial_x^2 z_k- \lambda z_k=0, 
    z_k\in L^2_k. 
\end{equation}
\textcolor{black}{With this property,
the stabilization result would follow} simply \textcolor{black}{by 
a} decomposition of $y$ in odd and even parts
\textcolor{black}{provided that the system is well-posed.}\\

\textcolor{black}{This existence of an isomorphism $T_{k}$ and a feedback law $K_{k}$ is given by the following key proposition.
}
\begin{prop}\label{thm-ope-t}
 Let the countable set 
\begin{equation}
    \mathcal{N}:= \{i^2-j^2: i, j\in \N\},
\end{equation}
let $m\in\mathbb{R}$ and 
$k\in\{1,2\}$. Assume that the sequence $\{a_n^{k}\}_{n}$ satisfies
\begin{gather}
    c n^{-m}< |a_n^k|< C n^{-m},\;\text{ for }k\in\{1,2\},  \;\text{ for }n\in\N^*,\\
    a^2_0\neq 0.
\end{gather}
Then for any $\lambda\notin \mathcal{N}$, there exists a sequence $\{K_n^{k}\}_{n}$ satisfying 
\begin{align*}
K^2_0 & \neq 0, \\
     c n^{m}< |K_n^{k}|< C n^{m}, & \;\text{ for }k\in\{1,2\},  \;\text{ for }n\in\N^*,\\
     \{(\lambda+ a_n^{k} K_n^{k}) n^r\}_{n}\in\ell^2, & \;\; \forall r\in [0, 1/2),\\
     K_{k} \textrm{ is a bounded functional } & \textrm{on } H^{m+1/2+}_k, 
\end{align*}
such that the linear operator $T_{k}$ defined as follows 
\begin{align}
    T_{k}: & \; \mathcal{S} \rightarrow \mathcal{S}'_k, \\
     f_n^{k} & \mapsto -K_n^{k} \sum_{p} \frac{a_p f_p^{k}}{p^2+\lambda-n^2}, \\
       f_n^{3-k} & \mapsto 0,
\end{align}
can be linearly extended to $H^{m-3/2+}$, and
\begin{gather}
    T_{k} \textrm{ is an isomorphism on } H^{m+\textcolor{black}{s}}_k \textrm{ for any } \textcolor{black}{s}\in (-3/2, 3/2),\\
    T_{k}\textcolor{black}{\phi_{k}}= \textcolor{black}{\phi_{k}} \textrm{ in } H^{m-1/2-}_k, 
\end{gather}
and moreover, for any $\textcolor{black}{r}\in (-1/2, 1/2)$, for any $\varphi\in H^{m+\textcolor{black}{r}+1}_k$ we have that 
\begin{equation}\label{eq:operator-eq}
      (T_{k}A+ T_{k}\phi_{k}K_{k})\varphi= (AT_{k}-\lambda T_{k})\varphi \textrm{ in } H^{m+ \textcolor{black}{r}-1}_{k}.
\end{equation}
\end{prop}
{\color{black}
\textcolor{black}{This proposition gives exactly what we want. Indeed,} if we denote by $$B= (\phi_1, \phi_2), \; K= (K_1, K_2)^T $$ 
and the linear operator $T$ by
\begin{gather}
    T_{12} f:= T_1 f+ T_2 f, \forall f\in \mathcal{S},
\end{gather}
then immediately we get \textcolor{black}{the following}.
\begin{cor}\label{cor-key}
 Under the assumption of Proposition \ref{thm-ope-t}, the transformation $T_{12}$ can be linearly extended on $H^{m-3/2+}$. Moreover, 
\begin{equation*}
    T_{12} \textrm{ is an isomorphism on } H^{m+s} \textrm{ for any } s\in (-3/2, 3/2),
\end{equation*}
and, 
for any $r\in (-1/2, 1/2)$, for any $\varphi\in H^{m+r+1}$,
\begin{gather}
 (T_{12}A+ T_{12}BK)\varphi= (AT_{12}-\lambda T_{12})\varphi \;  \textrm{ in } H^{m+r-1}, \label{eq:operatorequality}\\
    T_{12}B=B \; \textrm{ in } H^{m-1/2-}.\label{eq:TB=B}
\end{gather} 
\end{cor}
}
Proposition \ref{thm-ope-t} will be proved in Section \ref{sec:proof-prop}.

\begin{remark}\label{rmk:regularityoperatorequality}
The functional setting of \eqref{eq:operatorequality} is optimal without assuming additional compatibility conditions. Indeed, as shown in \cite{CGM}, the operator equality \eqref{eq:operatorequality} is satisfied for $r=1$ if $\phi \in H^{m+r+1}$ satisfies additional compatibility conditions, {\color{black} for $\varphi$ satisfying some regularity requirement, namely $D(A+ BK)$}. In \cite{CGM}, this can be seen from the fact that $r\geq 1/2$ is the precise space for which the trace of $\phi$ makes sense and for which the compatibility conditions can be ensured. 
\end{remark}

\section{Building the backstepping transformation}
\label{sec:proof-prop}
In this section we only work on the odd functions which correspond to the $H^m_1$ spaces. We assume $k=1$ in the following.
Similar results hold for even functions which correspond to the $H^m_2$ spaces.
To simplify the notations, in this section we ignore the index $k$ and we also denote \textcolor{black}{$T:= T_k, K:= K_k$ and $\phi:=\phi_{k}$.} This Section will be divided in several parts: 
\begin{enumerate}
    \item In Section \ref{sec-set-up-back} we reformulate the problem by projecting the equation on the eigenfunctions of $A$ and we build a candidate $T$ to satifies the operator equality \eqref{eq:operator-eq}.
    \item In Section \ref{sec-riesz-pro} we prove some Riesz basis properties which will be used to study the invertibility of $T$.
    \item In \ref{sec-feedback-candidate} we prove the existence of a feedback $K$ such that the corresponding candidate $T$ satisfies the condition $T\phi = \phi$ weakly.
    \item In Section \ref{sec-T-inv-L2} we show that this candidate $T$ is an isomorphism that satisfies the operator equality \eqref{eq:operator-eq}.
\end{enumerate}

\subsection{Setting up for the backstepping transformation}\label{sec-set-up-back}

We want to map the solution of 
\begin{equation}
    y_t-y_{xx}= \phi Ky, \quad y\in L^2_1,
\end{equation}
via transformation $T$, to the solution of
\begin{equation}
    z_t-z_{xx}- \lambda z=0, \quad z\in L^2_1.
\end{equation}

To achieve this aim we would like $T$ to satisfies formally the backstepping conditions 
\begin{align}
& TA + \phi K = AT - \lambda T \label{formal--1}\\
& T\phi=\phi
\end{align}
for a suitable feedback law $K$:
\begin{equation}
    K: f_n\mapsto \langle  f_n, K\rangle=:K_n\in \R.
\end{equation}
Projected on the eigenvectors $f_{n}:=\sin nx$ with eigenvalues $\lambda_n=-n^2$, the formal relation \eqref{formal--1} becomes
\begin{gather}
T\textcolor{black}{(}\Delta f_{n}\textcolor{black}{)}+\langle f_{n},K\rangle \phi=\Delta(Tf_{n})-\lambda (Tf_{n}),\label{eq:firstcond}\\
\langle T\phi,f_{n}\rangle = \langle\phi,f_{n}\rangle.
\label{eq:secondcond}
\end{gather}

Defining \begin{equation}
    h_n:=Tf_{n},
\end{equation}  the first condition becomes
\begin{equation}
\lambda_{n} h_n+\langle f_{n},K\rangle \phi=\Delta h_n-\lambda h_n.
\end{equation}
Projecting the preceding equation now on $f_{p}$, defining $$a_{p}:= \langle \phi,f_{p}\rangle$$ and using the fact that $\Delta$ is self-adjoint we get
\begin{equation}
\begin{split}
\lambda_{n} \langle h_n,f_{p}\rangle+\langle f_{n},K\rangle a_{p}=(\lambda_{p}-\lambda)\langle h_n,f_{p}\rangle,
\end{split}
\end{equation}
Hence, for any $n, p\in \N^*$, 
\begin{equation}
    \langle h_n,f_{p}\rangle=\frac{-K_n a_{p}}{\lambda_{n}-\lambda_{p}+\lambda}
\end{equation}

Therefore 
\begin{equation}
    q_n:= - \frac{h_n}{K_n}= \sum_{p\in \N^*} \frac{a_p f_p}{\lambda_{n}-\lambda_{p}+\lambda}, \quad \forall n\in \N^*.
\end{equation}
Inspired by the preceding formula, the number $\lambda$ should be selected in such a way that 
\begin{equation}
    \lambda_{n}-\lambda_{p}+\lambda\neq 0, \quad \forall p, n\in \N^*,
\end{equation}
which is rather easy to achieve, for example to choose from $\N^*+ 1/2$. More precisely, it suffices to choose 
\begin{equation}
    \lambda\notin \mathcal{N}:= \{i^2-j^2: i, j\in \N\}.
\end{equation}

\subsection{Riesz basis properties}\label{sec-riesz-pro}
Recall that for any $n\in \N^*$,
\begin{gather}
    f_n= \dfrac1{\sqrt{\pi}}\sin nx, \; \; \lambda_n=- n^2, \\
    A:= \Delta, \;\;\; \phi= \sum_{n\in\mathbb{N}^{*}} a_n f_n, \\ h_n:= T f_n, \;\;\;  -K_n q_n:= h_n, \\
       q_n=  \sum_{p\in \N^*} \frac{a_p f_p}{\lambda_{n}-\lambda_{p}+\lambda}, \\
     g_n:=  \sum_{p\in \N^*} \frac{f_p}{\lambda_{n}-\lambda_{p}+\lambda}\in H^{3/2-}.
\end{gather}
\textcolor{black}{This last claim on the regularity of the $g_n$ comes from the growth of the eigenvalues $\lambda_{p}$,} {\color{black} \textit{i.e.,} for any $n\in \N^*$ we have 
\begin{gather}
  \|g_n\|_{H^{s}}^2=  \sum_{p\in \N^*} \frac{p^{2s}}{(p^2-n^2+\lambda)^2}< +\infty, \forall s\in (-\infty, 3/2), \\
   \|g_n\|_{H^{s}}^2=  \sum_{p\in \N^*} \frac{p^{2s}}{(p^2-n^2+\lambda)^2}= +\infty, \textrm{ for } s=3/2.
\end{gather}

Notice that $\{a_n\}_{n\in \N^*}$ is uniquely determined by the value of the function $\phi$, while the sequences $\{g_n\}_{n\in \N^*}$ and   $\{q_n\}_{n\in \N^*}$ are independent of the choice of  $\{K_n\}_{n\in \N^*}$. Hence any  sequence  $\{K_n\}_{n\in \N^*}$ determines the value of  $\{
h_n\}_{n\in \N^*}$, thus the operator $T$, and such operator $T$ (at least formally) satisfies the equation \eqref{formal--1}. 

\subsubsection{Existence}
The following lemma is devoted to the properties of  $\{g_n\}_{n\in \N^*}$,   to the properties of   $\{q_n\}_{n\in \N^*}$ provided some suitable assumption on $\{a_n\}_{n\in \N^*}$, and to the properties of the transformation $T$ provided some assumption on both $\{a_n\}_{n\in \N^*}$ and $\{K_n\}_{n\in \N^*}$.  }

\begin{lem}
\label{lem-rbheat-1} Let $m\geq 0$. Let $a_n\neq 0$ such that $c n^{-m}< |a_n|<C n^{-m}$.  Let $\lambda\notin \mathcal{N}$. The following properties hold:
\begin{itemize}
    \item [(1)] $\{g_n\}_{n\in \N^*}$ is a Riesz basis of $L^2_1$.\\
    
 \item [(2)] Let $s\in (-3/2, 3/2)$. Then 
$\{n^{-s} g_n\}_{n\in \N^*}$ is a Riesz basis of $H^{s}_1$.\\
     
      \item [(3)] 
      Let $s\in (-3/2, 3/2)$. Then 
$\{n^{-s} q_n\}_{n\in \N^*}$ is a Riesz basis of $H^{m+s}_1$.\\
      
      \item[(4)] Let $m\geq 0$. Let  $s\in (-3/2, 3/2)$.  If $K_n:= \langle f_{n},K\rangle$ is chosen in such a way that $|K_n|<  C n^m$, then   the transformation $T:H^{m+s}_1\rightarrow H^{m+s}_1$ is bounded. \\
      Moreover, if $c n^m <|K_n|<  C n^m$, then the transformation $T:H^{m+s}_1\rightarrow H^{m+s}_1$ is an isomorphism.\\
     \end{itemize}
     Let us remark here that all the choices of $s$ and $r$ in the above are sharp.
     \end{lem}
    
We give here the key ideas of the proof. The rigorous proof itself is detailed below. 
\begin{itemize}
\item In order to show that $\{g_n\}_{n\in \N^*}$ is a Riesz basis of $L^2_1$ we show first that it is quadratically close to $\{f_{n}/\lambda\}_{n\in \N^*}$ --which is obviously an orthonormal basis of $L^{2}_{1}$-- and then we show that $\{g_n\}_{n\in \N^*}$ is either $\omega$-independent or dense in $L^{2}_{1}$ and we conclude using Lemma \ref{lem-cri-riesz}. Showing that $\{g_n\}_{n\in \N^*}$  is quadratically close to $\{f_{n}/\lambda\}_{n\in \N^*}$ amounts to show that
\begin{equation}\label{quaappro0}
    \left(\sum_{p\in \N^*} \frac{f_p}{\lambda_n-\lambda_p+\lambda}\right)_{n\in \N^*} \textrm{ is quadratically close to } \left(\frac{f_n}{\lambda}\right)_{n\in \N^*}\text{ in }L_{1}^{2}.
\end{equation}
For this, it suffices to show that
\begin{equation}\label{eq:sumclose0}
    \sum_{n\in\mathbb{N}^{*}} \sum_{p\neq n} \left(\frac{1}{p^2+\lambda-n^2}\right)^2<+\infty,
\end{equation}
which can be done by cases (see \eqref{trace1}--\eqref{eq:445} below) by using that for $\lambda\in\mathbb{N}^{*}$ and any $n>\lambda$ and $p<n$ there exists $j\in\{1,...,n-1\}$ such that $p =( n-j)$ and 
\begin{equation}\label{eq:manipj0}
n^2-(n-j)^2-\lambda\geq (n^2-(n-j)^2)/2.
\end{equation}
One can look at \eqref{eq:manipj}--\eqref{end-manip} for more details,  and Remark \ref{rem:lambda-np} for general value of $\lambda$.
\item Showing that $\{g_n\}_{n\in \N^*}$ is either $\omega$-independent or dense in $L^{2}_{1}$ can be done by noticing first that
    \begin{equation}
    \mathcal{A}^{-1}g_n= \Delta^{-1}g_{n}= (n^2-\lambda)^{-1} g_n-(n^2-\lambda)^{-1}\mathcal{A}^{-1}h,
\end{equation}
and assuming that $\{g_n\}_{n\in \N^*}$ is not $\omega$-independent (otherwise the proof is done). Then, we deduce the existence of 
\begin{equation}
    \sum_{n\in \N^*} c_n g_n=0, \textrm{ in } L^2_1.
\end{equation}
The preceding formula is well-defined, in fact, thanks to \eqref{quaappro0},
\begin{equation}
    \sum_{n\in \N^*} c_n g_n=  \sum_{n\in \N^*} c_n \frac{f_n}{\lambda}+ \sum_{n\in \N^*} c_n \left(g_n- \frac{f_n}{\lambda}\right)
\end{equation}
converges in $L^2_1$.
Next,  by applying $\mathcal{A}^{-1}$ to this equality we conclude
\begin{equation}
    \sum_{n\in \N^*} c_n k_n g_n= \sum_{n\in \N^*} c_n k_n \mathcal{A}^{-1} h, \textrm{ in } L^2_1,
\end{equation}
where we have used the fact $\sum_n c_n k_n$ converges. We then iterate and show (see \eqref{eq:448}--\eqref{eq:452} for more details) that
\begin{equation}
      \sum_{n\in \N^*} c_n k_n^m g_n= \sum_{i=1}^m C_{m+1-i}  \mathcal{A}^{-i} h,
\end{equation}
with
\begin{equation}
\label{defCl}
    C_{l}:= \sum_{n\in \N^*} c_n k_n^{l}< +\infty.
\end{equation}
From this point, there are only two possibilities: either there exists $m\geq 1$ such that $C_{m}\neq 0$ and we can show that $\{g_n\}_{n\in \N^*}$ is dense (see \eqref{eq:455}--\eqref{spangndense} below concerning what is referred as ``First case" and ``Second case"), or $C_{m}= 0$ for any $m\in\mathbb{N}^{*}$ and we get contradiction using that the complex function
\begin{equation}
    \tilde{G}(z) =\sum_{n\in \N^*} c_n k_n e^{k_{n}z}
\end{equation}
is holomorphic (and hence identically equal to 0 from \eqref{defCl}).
\item Showing that   $\{n^{-s} g_n\}_{n\in \N^*}$ is a Riesz basis of $H^{s}_1$ is  done in a similar fashion as the $s=0$ case. Then, showing that $\{n^{-s}q_{n}\}_{n\in\mathbb{N}^{*}}$ is a Riesz basis of $H_{1}^{m+s}$ is done by using that
\begin{equation}
    \tau: n^{-s}f_{n}\rightarrow n^{-s} a_{n}f_{n}
\end{equation}
is an isomorphism from $H^{s}$ to $H^{s+m}$ mapping $g_{n}$ to $q_{n}$.
\item Finally, assuming that $c n ^{m}<K_{n}<C n^{m}$, we know that $\{\frac{n^{-s}f_n}{n^m}\}_{n\in \N^*}$ is an orthonormal basis of $H^{m+s}_1$ and because $\{n^{-s} q_n\}_{n\in \N^*}$ is a Riesz basis of $H^{m+s}_1$, we deduce from the assumptions on $K_{n}$ and the previous point that $\{\frac{-K_n}{n^m} (n^{-s}q_n)\}_{n\in \N^*}$ is also a Riesz basis of $H^{m+s}_1$. From Lemma \ref{lem-iso-riesz-equi} this means that $T$ is an isomorphism from $H_{1}^{m+s}$ to itself.
\end{itemize}

\begin{proof}[Proof of Lemma \ref{lem-rbheat-1}]

For ease of notations, in the proof of this lemma  we fix $$\lambda=N= 4M+2,$$ which guarantees the fact that $p^2+\lambda-n^2\neq 0$. However, all the results hold with similar calculation for any $\lambda$ outside the special subset $\mathcal{N}$, which in particular is guaranteed by Remark \ref{rem:lambda-np} stated below.\\

\noindent {\bf \textcolor{black}{(1) $\{g_n\}_{n\in \N^*}$ is a Riesz basis of $L^2_1$.}}
\textcolor{black}{We proceed in two steps.}
We start by showing  that $\{g_n\}_{n\in \N^*}$ is quadratically close to a Riesz basis in $L^{2}_{1}$. \textcolor{black}{Then we show that it is $\omega$-independent or  dense in $L^{2}_{1}$, which, together with the quadratically close behavior, ensures that it is a Riesz basis of $L^{2}_{1}$.}
\begin{equation}\label{quaappro}
    \left(\sum_{p\in \N^*} \frac{f_p}{\lambda_n-\lambda_p+\lambda}\right)_{n\in \N^*} \textrm{ is quadratically close to } \left(\frac{f_n}{\lambda}\right)_{n\in \N^*}\textcolor{black}{\text{ in }L_{1}^{2}}.
\end{equation}
It suffices to show that 
\begin{equation}\label{eq:sumclose}
    \sum_{n\in\mathbb{N}^{*}} \sum_{p\neq n} \left(\frac{1}{p^2+\lambda-n^2}\right)^2<+\infty.
\end{equation}

Thus it further suffices to prove  \begin{equation}\label{ine-nN}
    \sum_{n>N} (\sum_{p> n}+ \sum_{p< n}) \left(\frac{1}{p^2+\lambda-n^2}\right)^2<+\infty.
\end{equation}
as well as
\begin{equation}\label{trace1}
   I:=  \sum_{n\leq N} (\sum_{p> n}+ \sum_{p< n}) \left(\frac{1}{p^2+\lambda-n^2}\right)^2<+\infty:
\end{equation}

We can express $I$ in the following fashion,
\begin{equation}
    I= \sum_{j= p^2+\lambda-n^2, \; n<N} \frac{1}{j^2}, \textrm{ counting multiplicity of } j,
\end{equation}
for any possible $j$ the multiplicity count at most as N, thus 
\begin{equation}
    I< N \sum \frac{1}{j^2}<+\infty.
\end{equation}

For the first part of \eqref{ine-nN} as $p>n$, we have  
\begin{align*}
     \sum_{n>N} \sum_{p> n} \left(\frac{1}{p^2+\lambda-n^2}\right)^2&\leq  \sum_{n=N} \sum_{k=1 } \left(\frac{1}{k^2+ 2kn}\right)^2, \\
     &\leq  \sum_{n=1} \sum_{k=1 } \left(\frac{1}{kn}\right)^2<+\infty.
\end{align*}

For the second part as $p<n$, \textcolor{black}{
there exists $j\in\{1,...,n-1\}$ such that $p=(n-j)$ and} we know that \textcolor{black}{for any such $j$}
\begin{equation}\label{eq:manipj}
n^2-(n-j)^2-\lambda\geq (n^2-(n-j)^2)/2.
\end{equation}
\textcolor{black}{Indeed, using that $\lambda =N$.
\begin{equation}
\begin{split}
    n^2-(n-j)^2-\lambda- (n^2-(n-j)^2)/2&=n^2/2-(n-j)^2/2-N,\\
    &\geq jn-j^{2}/2-N,
\end{split}
\end{equation}
and the right-hand side is a second order polynomial whose minimum is achieved either for $j=1$ or for $j=n-1$. As $n\geq N+1$,
\begin{equation}
\label{end-manip}
\begin{split}
        n^2-(n-j)^2-\lambda- (n^2-(n-j)^2)/2&\geq 1/2\\
        &\geq 0.
        \end{split}
\end{equation} }

\begin{remark}\label{rem:lambda-np}
For the general case where $\lambda\notin \mathcal{N}$, there exists $C(\lambda)>0$ such that 
\begin{equation}
    |p^2+\lambda-n^2|\geq C(\lambda)|p^2- n^2|, \; \forall p, n\in \N^*.
\end{equation}
Indeed, as for $p\geq n$ we always have  $|p^2+\lambda-n^2|\geq |p^2- n^2|$, it suffices to consider the case that $p\leq n-1$. Thus, it is  equivalent to show that for $p\leq n-1$, 
\begin{equation}\label{ine-pnlam-c}
    \frac{ |p^2+\lambda-n^2|}{n^2-p^2}> C(\lambda).
\end{equation}
For $n\geq \lambda+1$ and $p\leq n-1$, we know that 
\begin{equation}
    \frac{ |p^2+\lambda-n^2|}{n^2-p^2}= \frac{ n^2-p^2-\lambda}{n^2-p^2}>  \frac{ n^2-p^2-n}{n^2-p^2}>  \frac{1}{2}.
\end{equation}
For $n\leq \lambda$ and $p\leq n-1$ containing finitely many pairs, thanks to the definition of $\lambda$, it is clear that such $C(\lambda)$ exists. 
\end{remark} 
Thus

\begin{align*}
 \sum_{n>N} \sum_{p< n} \left(\frac{1}{p^2+\lambda-n^2}\right)^2&=
     \sum_{n>N} \sum_{k=1}^{n-1} \left(\frac{1}{k^2+\lambda-n^2}\right)^2,  \\
     &\leq  4\sum_{n>N} \sum_{k=1}^{n-1} \left(\frac{1}{n^2-k^2}\right)^2, \\
    &\leq  4\sum_{n>N} \sum_{k=1}^{n-1} \left(\frac{1}{n(n-k)}\right)^2,\\
    &\leq  4\sum_{n>N} \sum_{\textcolor{black}{j}=1}^{n-1} \left(\frac{1}{n\textcolor{black}{j}}\right)^2< +\infty.
\end{align*}

Let us denote 
\begin{equation}\label{eq:445}
    g_n:= \sum_{p\in \N^*} \frac{f_p}{\lambda_n-\lambda_p+\lambda}=  \sum_{p\in \N^*} \frac{f_p}{-n^2+ p^2+\lambda}\in H^{3/2-}_{\textcolor{black}{1}}.
\end{equation}
Then
\begin{equation}\label{res:gn}
    (-\Delta + \lambda- n^2)g_n= \sum_{p\in \N^*} f_p= \frac{1}{2} \cot{\frac{x}{2}}=: h, \textrm{ in } H^{-1}_{\textcolor{black}{1}}.
\end{equation}

The following proof to show that $g_n$ is either $\omega$-independent or dense in $L^2_1$  is inspired by \cite{CoronLu14}, though even the transformation type that is adapted here is slightly different from the one given in \cite{CoronLu14}. 

Recalling that $\mathcal{A}:= -A= -\Delta$ and defining
$k_n:= (n^2-\lambda)^{-1}$,  from \eqref{res:gn} we notice that 
\begin{equation}
    \mathcal{A}^{-1}g_n= k_n g_n-k_n \mathcal{A}^{-1}h.
\end{equation}

If  $\{g_n\}_{n\in \N^*}$ is  $\omega$-independent then we conclude the proof.  Suppose that $\{g_n\}_{n\in \N^*}$ is not $\omega$-independent, thus by the definition there exists a nontrivial sequence $\{c_n\}_{n\in \N^*}$ belongs to $\textcolor{black}{l^2}(\N^{*})$ 
such that 
\begin{equation}\label{eq:448}
    \sum_{n\in \N^*} c_n g_n=0, \textrm{ in } L^2_1.
\end{equation}
The preceding formula is well-defined, in fact, thanks to \eqref{quaappro} that we just proved,
\begin{equation}
    \sum_{n\in \N^*} c_n g_n=  \sum_{n\in \N^*} c_n \frac{f_n}{\lambda}+ \sum_{n\in \N^*} c_n \left(g_n- \frac{f_n}{\lambda}\right)
\end{equation}
converges in $L^2_1$.

Next,  by applying $\mathcal{A}^{-1}$ to this equality we conclude
\begin{equation}
    \sum_{n\in \N^*} c_n k_n g_n= \sum_{n\in \N^*} c_n k_n \mathcal{A}^{-1} h, \textrm{ in } L^2_1,
\end{equation}
where we have used the fact $\sum_n c_n k_n$ converges.

Then, by applying again $\mathcal{A}^{-1}$ to this equality we get that in $L^2_1$ space,
\begin{equation}
      \sum_{n\in \N^*} c_n k_n^2 g_n= \sum_{n\in \N^*} c_n k_n^2 \mathcal{A}^{-1} h+ c_nk_n \mathcal{A}^{-2}h,
\end{equation}
and applying it again we have still in $L^2_1$ space,
\begin{equation}
      \sum_{n\in \N^*} c_n k_n^3 g_n= \sum_{n\in \N^*} c_n k_n^3 \mathcal{A}^{-1} h+ c_n k_n^2 \mathcal{A}^{-2}h+ c_n k_n \mathcal{A}^{-3}h.
\end{equation}
By induction we easily arrive at, for any $m\in \N^*$,
\begin{equation}
      \sum_{n\in \N^*} c_n k_n^m g_n= \sum_{i=1}^m \left(\sum_{n\in \N^*} c_n k_n^{m+1-i} \mathcal{A}^{-i} h\right)= \sum_{i=1}^m C_{m+1-i}  \mathcal{A}^{-i} h,
\end{equation}
where 
\begin{equation}\label{eq:452}
    C_{l}:= \sum_{n\in \N^*} c_n k_n^{l}< +\infty.
\end{equation}
\textcolor{black}{Let us now proceed by cases:}

\textcolor{black}{\textbf{- First case: $C_1\neq 0$}. T}hen we conclude from the proceding equation that for all $m\in \N^*$ we have that  $A^{-\textcolor{black}{m}}h\in $ span$\{g_n\}_{n\in \N^*}$. Suppose that  span$\{g_n\}_{n\in\mathbb{N}^{*}}$ is not dense in $L^2_1$, then we can find $d=\sum d_n f_n\in L^2_1$ (thus $\{d_n\}_{n\in \N^*}\in\ell^2(\N^*) $) such that \textcolor{black}{$d\neq 0$ and }
\begin{equation}\label{eq:455}
    \langle g, d\rangle_{L^2_1}=0,\;  \forall g\in \textrm{span} \{g_n\}_{n\in\mathbb{N}^{*}},
\end{equation}
which in particular yields,
\begin{equation}
    \langle \mathcal{A}^{-\textcolor{black}{m}}h, d\rangle_{L^2_1}=0, \forall \textcolor{black}{m}\in \N^*.
\end{equation}
Recalling that $h=\sum f_n\in H^{-1}_{\textcolor{black}{1}}$, we get that 
\begin{equation}\label{dnsum0}
   \sum_n d_n \frac{1}{n^{2\textcolor{black}{m}}}=0, \forall \textcolor{black}{m}\in \N^*.
\end{equation}
By defining the complex function 
\begin{equation}
    G(z):= \sum_{n\in \N^*} d_n n^{-2} e^{n^{-2} z}, \forall z\in \mathbb{C}.
\end{equation}
By checking that the series expansion of the right-hand side is absolutely convergent, we deduce that this function is holomorphic. For example,   for any $z\in \mathbb{C}$, the following series is absolutely convergent,
\begin{align*}
   \sum_{n\in \N^*} d_n n^{-2} e^{n^{-2} z}=  \sum_{n\in \N^*} d_n n^{-2} \sum_{j\geq 0} \frac{n^{-2j}}{j!} z^j= \sum_{j\geq 0} \frac{1}{j!}\sum_{n\in \N^*} (d_n n^{-2}) (n^{-2j} z^j).
\end{align*}
Similar calculation yields the absolute convergence of $G^{\textcolor{black}{(m)}}(z)$, \textcolor{black}{for any $m\in\mathbb{N}$}.

From \eqref{dnsum0} we know that $G^{\textcolor{black}{(m)}}(0)= 0, m\in \N$. Thus $G=0$, and further $d_n=0$, which leads to a contradiction. Therefore
\begin{equation}\label{spangndense}
   \textrm{span} \{g_n\}_{n\in \N^*} \textrm{ is  dense in }L^2_1.
\end{equation}

\textcolor{black}{\textbf{- Second case: there exists $m>1$ such that $C_m\neq 0$}, without loss of generality we can assume that $m$} is the first integer such that $C_m\neq 0$. We can also conclude that   $A^{-m}h\in $ span$\{g_n\}_{n\in \N^*}$. The same reasoning as above proves again \eqref{spangndense}.\\

\textcolor{black}{\textbf{- Third case: for all $l\in \N^*$ we have $C_l=0$.} T}hen we set complex function
\begin{equation}
    \tilde{G}(z):= \sum_{n\in \N^*} c_n k_n e^{k_n z}.
\end{equation}
This function is holomorphic and satisfies that $\tilde{G}^{(m)}(0)=0$ \textcolor{black}{for any $m\in\mathbb{N}$, thus as previously $\tilde{G}= 0$ and therefore} \textcolor{black}{$c_n=0$ for all $n\in\mathbb{N}^{*}$}, which is in contradiction with the choice of $\{c_n\}_{n\in \N^*}$.

Consequently, 
\begin{equation}\label{comorind}
   \textrm{$\{g_n\}$ is either $\omega$ independent in $L^2_1$ or  dense in  $L^2_1$.} 
\end{equation}

Properties \eqref{quaappro} and \eqref{comorind}, together with Lemma \ref{lem-cri-riesz} (which is also \cite[Theorem 3.2 and Theorem 3.3]{CGM}), lead to the proof of Lemma \ref{lem-rbheat-1} (1).\\

\noindent {\bf (2) \textcolor{black}{$\{n^{-s} g_n\}_{n\in \N^*}$ is a Riesz basis of $H^{s}_1$.}} 

\textcolor{black}{As in (1),} we first show that $\{n^{-s} g_n\}_{n\in \N^*}$ is quadratically close to $\{\frac{n^{-s} f_n}{\lambda}\}_{n\in \N^*}$ in $\textcolor{black}{H^{s}_{1}}$, and then prove that  $\{n^{-s}g_n\}_{n\in \N^*}$ is $\omega$-independent in $H^s_1$ or dense in  $H^s_1$  to conclude the proof.\\

{\bf $\star$ Quadratically close.}

We notice that $s=0$ is exactly the case of {\bf (1)}.

i) If $s=-1$ then 
\begin{align*}
    \sum_{n}\left\|n g_n-\frac{n f_n}{\lambda}\right\|_{H^{-1}}^2&=  \sum_{n}\left\|\sum_{p\neq n} \frac{n f_p}{p^2+\lambda-n^2}\right\|_{H^{-1}}^2, \\
    &= \sum_{n}\sum_{p\neq n} \frac{n^2 }{p^2(p^2+\lambda-n^2)^2}, \\
     &= \sum_{n}\left(\sum_{p< n}+ \sum_{p> n}\right) \frac{n^2 }{p^2(p^2+\lambda-n^2)^2}, \\
     &\leq C+ \sum_{n}\sum_{p< n}\frac{n^2 }{p^2(p^2+\lambda-n^2)^2}, \\
      &\leq C+ \sum_{n\geq N}\sum_{p< n}\frac{n^2 }{p^2(n^2-p^2-\lambda)^2}, \\
       &\leq C+ 4\sum_{n\geq N}\sum_{p< n}\frac{n^2}{p^2(n^2-p^2)^2}, \\
       &\leq C+ C \sum_{n\geq N}\sum_{p< n}\frac{n^2 }{p^2n^2(n-p)^2 } , \\
        &\leq C+ C \sum_{n\geq N}\sum_{p< n}\frac{1}{n^2} \left(\frac{1}{p}+ \frac{1}{n-p}\right)^2, \\
         &\leq C+ C \sum_{n\geq N}\sum_{p< n}\frac{1}{n^2} \left(\frac{1}{p^2}+ \frac{1}{(n-p)^2}\right),\\
          &\leq C+ C \sum_{n\geq N}\sum_{p< n}\frac{1}{p^2n^2}  <+\infty,
\end{align*}
\textcolor{black}{where $C$ is a constant that can change between lines and where used \eqref{eq:sumclose}, \eqref{eq:manipj}, the fact that $\lambda=N$, and
\begin{equation}
    \frac{1}{(n-p)^{2}p^{2}}=\left(\frac{1}{n}(\frac{1}{n-p}+ \frac{1}{p})\right)^2\leq \frac{2}{n^{2}}\left(\frac{1}{(n-p)^{2}}+\frac{1}{p^{2}}\right).
\end{equation}
}

ii) From the preceding proof we easily conclude the case $s\in (-1, 0)$ \textcolor{black}{by using that $(n/p)^{-s}\leq (n/p)$ for $0<p<n$}.

iii) If $s\in (-3/2, -1)$, it suffices to show that 
\[\sum_{n\geq N}\sum_{p< n}\frac{1 }{(n-p)^2 n^2}\frac{n^{-2s}}{p^{-2s}}<+\infty.\]
\textcolor{black}{since the case $p>n$ follows directly from \eqref{eq:sumclose} and $(n/p)^{-2s}\leq 1$.}
For $p<n/2$, we have 
\begin{align*}
    \sum_{n\geq N}\sum_{p< n/2}\frac{1 }{(n-p)^2 n^2}\frac{n^{-2s}}{p^{-2s}}&\leq C  \sum_{n\geq N}\sum_{p< n/2}\frac{1 }{ n^4}\frac{n^{-2s}}{p^{-2s}}\\
    &    \leq C \sum_{n\geq N}\frac{1 }{ n^{4+2s}}\sum_{p< n/2}\frac{1}{p^{-2s}}\\
    &\textcolor{black}{\leq C \left(\sum_{n\geq N}\frac{1 }{ n^{4+2s}}\right)\left(\sum_{p\in\mathbb{N}^{*}}\frac{1}{p^{-2s}}\right)}\\
    &<+\infty.
\end{align*}
For $n/2\leq p<n$, we have
\begin{equation}
     \sum_{n\geq N}\sum_{n/2\leq p< n}\frac{1 }{(n-p)^2 n^2}\frac{n^{-2s}}{p^{-2s}}\leq C  \sum_{n\geq N}\frac{1}{n^2}\sum_{n/2\leq p< n}\frac{1 }{(n-p)^2 }<+\infty.
\end{equation}

iii) If $s=\frac{3}{2}-\varepsilon$ with $\varepsilon\in (0, \frac{3}{2})$, then
\begin{align*}
     \sum_{n}\left\|n^{-s} g_n-\frac{n^{-s} f_n}{\lambda}\right\|_{H^{s}}^2&=  \sum_{n}\left\|\sum_{p\neq n} \frac{n^{-s} f_p}{p^2+\lambda-n^2}\right\|_{H^{s}}^2, \\
     &= \sum_{n}\sum_{p\neq n} \frac{n^{-2s} p^{2s} }{(p^2+\lambda-n^2)^2}, \\
     &= \sum_{n}\left(\sum_{p< n}+ \sum_{n<p<2 n}+ \sum_{p\geq 2n}\right)\frac{1}{(p^2+\lambda-n^2)^2} \frac{p^{2s}}{n^{2s}}, \\
      &\leq C+ \sum_{n} \sum_{p\geq 2n}\frac{1}{(p^2+\lambda-n^2)^2} \frac{p^{2s}}{n^{2s}}, \\
       &\leq C+ \sum_{n} \sum_{p\geq 2n}\frac{1}{(p^2-n^2)^2} \frac{p^{2s}}{n^{2s}},\\
        &\leq C+ C\sum_{n} \sum_{p\geq 2n}\frac{1}{(p^2)^2} \frac{p^{2s}}{n^{2s}},\\
        &\leq C+ C \sum_{n} \frac{1}{n^{2s}} \sum_{p\geq 2n}\frac{1}{p^{4-2s}}.
\end{align*}
Because $s\in (0, 3/2)$, we know that 
\[
\sum_{p\geq 2n}\frac{1}{p^{4-2s}}\leq C \frac{1}{n^{3-2s}}\]
\textcolor{black}{and therefore $\left(\frac{1}{n^{2s}} \sum_{p\geq 2n}\frac{1}{p^{4-2s}}\right)_{n\in\mathbb{N}^{*}}$} is absolutely convergent.  

This concludes the proof of the quadratically close behavior.\\

{\bf $\star$ Let $s\in (-3/2, 3/2)$. $\{n^{-s}g_n\}_{n\in \N^*}$ is $\omega$-independent in $H^s_1$ or dense in  $H^s_1$. }

Similar to the case that $s=0$.   Suppose that there exists $\{c_n\}_{n\in \N^*}\in\ell^2(\N^{*})$ such that 
\begin{equation}
    \sum_{n\in \N^*} \frac{c_n}{n^s} g_n=0, \textrm{ in } H^s_1,
\end{equation}
which is well defined as
\begin{equation}
    \sum_{n\in \N^*} c_n n^{-s} g_n=  \sum_{n\in \N^*} c_n n^{-s} \frac{f_n}{\lambda}+  \sum_{n\in \N^*} c_n n^{-s} \left(g_n-\frac{f_n}{\lambda}\right)
\end{equation}
converges in $H^s_1$ space thanks to the fact that $s\in (-3/2, 3/2)$, and the quadratically close result in $H^s_1$ that we just proved.
  By the same reasoning as the case $s=1$, we get
\begin{equation}
    \sum_{n\in \N^*} \frac{c_n}{n^s} k_n g_n= \sum_{n\in \N^*} \frac{c_n}{n^s} k_n \mathcal{A}^{-1} h, \textrm{ in } H^s_1,
\end{equation}
where we have used the fact $\sum_n \frac{c_n}{n^s} k_n$ converges for $s>-3/2$.

Then, we can further consider 
\begin{equation}
    C_{m}:= \sum_{n\in \N^*} \frac{c_n}{n^s} k_n^{m}<+\infty.
\end{equation}
Exactly the same as in the case $s=0$, it suffices to consider two cases, $C_1\neq 0$, or $C_m=0$ \textcolor{black}{for all $m\in\mathbb{N}^{*}$}.

In the first case, we find that  $\mathcal{A}^{-m}h\in $ span$\{n^{-s} g_n\}_{n\in \N^*}$ in $H^s_1$. Suppose that  span$\{n^{-s}g_n\}_{n\in \N^*}$ is not dense in $H^s_1$, then we can find $d=\sum_n \frac{d_n}{n^s} f_n\in H^s_1$ (thus $\{d_n\}_{n\in \N^*}\in\ell^2(\N^*) $) such that 
\begin{equation}
    \langle \mathcal{A}^{-m}h, d\rangle_{H^s_1}=0,\; \forall m\in \N^*,
\end{equation}
which is well-defined as $h\in H^{-1/2-}_1$ and $\mathcal{A}^{-1}h\in H^{3/2-}_1\subset H^s_1$.
Recalling the exact definition of $h=\sum_n f_n\in H^{-1/2-}_1$, we get that 
\begin{equation}
   \sum_n d_n \frac{n^s}{n^{2m}}=0, \forall m\in \N^*,
\end{equation}
which further implies that $d_n=0$ using the holomorphic function technique.

In the second case where $C_m=0$ \textcolor{black}{for all $m\in\mathbb{N}^{*}$}, we can also prove that $c_n=0$, \textcolor{black}{similarly as in the case $s=0$}, which is in contradiction with the choice of $\{c_n\}_{n\in \N^*}$.
\\

\noindent {\bf (3) \textcolor{black}{$\{n^{-s} q_n\}_{n\in \N^*}$ is a Riesz basis of $H^{m+s}_1$.}} 

Let $s\in (-3/2, 3/2)$.  
We introduce $\tau: H^s_1\rightarrow H^{m+s}_1$ \textcolor{black}{defined by}, 
\begin{equation}
    \tau: n^{-s}f_n\mapsto n^{-s} a_n f_n,
\end{equation}
\textcolor{black}{which is an isomorphism} thanks to the fact that $c n^{-m}< |a_n|< C n^{-m}$. We immediately notice that 
\begin{equation}
    \tau: n^{-s} g_n=n^{-s} \sum_{p\in \N^*} \frac{f_p}{\lambda_n-\lambda_p+\lambda} \mapsto n^{-s} \sum_{p\in \N^*} \frac{a_p f_p}{\lambda_n-\lambda_p+\lambda} =n^{-s}q_n.
\end{equation}
Consequently, thanks to the fact that $\{n^{-s} g_n\}_{n\in \N^*}$ is a Riesz basis of $H^s_1$, from Lemma \ref{lem-iso-riesz-equi} we know that $\{n^{-s} q_n\}_{n\in \N^*}$ is a Riesz basis of $H^{m+s}_1$.\\

\noindent {\bf (4) \textcolor{black}{Boundedness of the transformation $T$ and isomorphism}}

If $K_n \neq 0$ is chosen in such a way that $c n^m< |K_n|< n^m$, then from the definition of $T$ we know that 
\begin{equation}
    T: \frac{n^{-s}f_n}{n^m}\mapsto \frac{-K_n}{n^m} (n^{-s}q_n).
\end{equation}

Since $\{n^{-s}f_n\}_{n\in \N^*}$ is an orthonormal basis of $H^s$, we know that $\{\frac{n^{-s}f_n}{n^m}\}_{n\in \N^*}$ is an orthonormal basis of $H^{m+s}_1$. Moreover, because $\{n^{-s} q_n\}_{n\in \N^*}$ is a Riesz basis of $H^{m+s}_1$, from Lemma \ref{lem-iso-riesz-equi} we know that  $\{\frac{-K_n}{n^m} (n^{-s}q_n)\}_{n\in \N^*}$ is also a Riesz basis of $H^{m+s}_1$. \\
Therefore, the transformation $T$
is an isomorphism from $H^{m+s}_1$ to $H^{m+s}_1$. \textcolor{black}{
When $m=2$, we remark, however, that $T$ is not a isomorphism from $L^2_1$ to itself (indeed $2+s>0$ for any $s\in (-3/2,3/2)$)}.\\

Moreover, if we only assume that $|K_n|\leq C n^m$, then $T$ is a bounded operator from $H^{m+s}_1$ to $H^{m+s}_1$. However, in this case $T$ might not be an isomorphism.\\
\end{proof}
\subsubsection{Smoothing properties}
In this section, we prove some smoothing properties which are not deduced trivially from the Riesz basis properties shown in Lemma \ref{lem-rbheat-1}.

\begin{lem}\label{lem-rbheat-2}
     \begin{itemize}
      \item [(1)] Let $r\in [0, 1/2)$. Then, $q_n$ has the following smoothing property, 
      \begin{equation}
          \sum_{n\in\mathbb{N}^{*}} \|q_n- a_nf_n/\lambda\|_{H^{m+r}_1}^2< +\infty.
      \end{equation}
      
       \item [(2)] Let $r\in [0, 1/2)$.  Similar smoothing effect also holds in the space $H^{-1+m}$,
      \begin{equation}
          \sum_{n\in\mathbb{N}^{*}} \|n(q_n- a_nf_n/\lambda)\|_{H^{-1+m+r}_1}^2< +\infty.
      \end{equation}
      
     \item [(3)] Let $m=0$ and $r\in [0, 1/2)$. Then,
 \begin{equation}
     \sum_{n\in\mathbb{N}^{*}} \left(q_n-\frac{a_n f_n}{\lambda}\right)\in H^r_1.
       \end{equation}
\end{itemize}
Again, all the choices of $s$ and $r$ in the above are sharp.
\end{lem}
\begin{proof}
\noindent {\bf (1) \textcolor{black}{Smoothing effect in $H^{m+r}_{1}$, for $r\in[0,1/2)$}} 
\begin{align*}
     \sum_n \left\|q_n- a_nf_n/\lambda\right\|_{H^{m+r}_1}^2&= \sum_n \left\|\sum_{p\neq n}\frac{a_pf_p}{p^2+\lambda-n^2}\right\|_{H^{m+r}_1}^2, \\
     &\leq \textcolor{black}{C}\sum_n \sum_{p\neq n}\frac{p^{2r}}{(p^2+\lambda-n^2)^2}, \\
     &= \textcolor{black}{C}\sum_p p^{2r}\sum_{n\neq p}\frac{1}{(p^2+\lambda-n^2)^2}, \\
     &\leq C +\textcolor{black}{C}\sum_{p\geq N} p^{2r}\sum_{n\neq p}\frac{1}{(p^2+\lambda-n^2)^2},
\end{align*}
\textcolor{black}{where C is a constant that can change between lines.}
Moreover, 
\begin{align*}
   \sum_{p\geq N} p^{2r}\sum_{n< p}\frac{1}{(p^2+\lambda-n^2)^2}&\leq   \sum_{p\geq N} p^{2r-2}\sum_{n< p}\frac{1}{(p-n)^2}<+\infty,
\end{align*}
and, \textcolor{black}{using \eqref{eq:manipj}}, 
\begin{align*}
   \sum_{p\geq N} p^{2r}\sum_{p<\textcolor{black}{n}}\frac{1}{(p^2+\lambda-n^2)^2}&\textcolor{black}{\leq \sum_{p\geq N} p^{2r}\sum_{p<\textcolor{black}{n}}\frac{4}{(n^{2}-p^{2})^2} }\\
   &\leq \sum_{p\geq N} p^{2r-2}\sum_{p<\textcolor{black}{n}}\frac{4}{(p-n)^2}<+\infty.
\end{align*}
From the above proof we find that the condition $r<1/2$ is sharp. \\

\noindent {\bf (2) \textcolor{black}{Smoothing effect in $H^{-1+m+r}_{1}$ for $r\in[0,1/2)$}} 
\begin{align*}
     \sum_n \left\|n(q_n- a_nf_n/\lambda)\right\|_{H^{-1+m+r}_1}^2&= \sum_n \left\|\sum_{p\neq n}\frac{n a_pf_p}{p^2+\lambda-n^2}\right\|_{H^{-1+m+r}_1}^2, \\
     &\leq \textcolor{black}{C} \sum_n \sum_{p\neq n}\frac{p^{2r-2} n^2}{(p^2+\lambda-n^2)^2}, \\
     &= \textcolor{black}{C} \sum_p p^{2r-2}\sum_{n\neq p}\frac{n^2}{(p^2+\lambda-n^2)^2}, \\
     &\leq C +\textcolor{black}{C} \sum_{p\geq N} p^{2r-2}\sum_{n\neq p}\frac{n^2}{(p^2+\lambda-n^2)^2}.
\end{align*}
Moreover, 
\begin{align*}
   \sum_{p\geq N} p^{2r-2}\sum_{n< p}\frac{n^2}{(p^2+\lambda-n^2)^2}&\leq   \sum_{p\geq N} p^{2r-2}\sum_{n< p}\frac{1}{(p-n)^2}<+\infty,
\end{align*}
and, \textcolor{black}{still using \eqref{eq:manipj}}, 
\begin{align*}
   \sum_{p\geq N} p^{2r-2}\sum_{p<\textcolor{black}{n}}\frac{n^2}{(p^2+\lambda-n^2)^2}&\leq   \sum_{p\geq N} p^{2r-2}\sum_{p<\textcolor{black}{n}}\frac{16}{(p-n)^2}<+\infty.
\end{align*}

\noindent {\bf (3) \textcolor{black}{The quantity $ \sum_{n\in\mathbb{N}^{*}} (q_n-\frac{a_n f_n}{\lambda})$ belongs to $H^r_1$ for $r\in[0,1/2)$.}}

At first we point out that since  $\{n q_n\}_{n\in \N^*}$ is a Riesz basis of $H^{-1}_1$, and $\{n f_n\}_{n\in \N^*}$ is an orthonormal basis of $H^{-1}_1$, the candidate  $ \sum_{n\in\mathbb{N}^{*}} (q_n-\frac{a_n f_n}{\lambda})$  belongs to $H^{-1}_1$.  We need to show that it actually belongs to more regular spaces $H^r_1$ for $r\in[0,1/2)$.   Moreover, both $ \sum_{n\in\mathbb{N}^{*}} q_n$ and $\sum_{n\in\mathbb{N}^{*}} \frac{a_n f_n}{\lambda}$  belong to $H^{-1/2-}$ but not to $H^{-1/2}$. 

In the following, we mainly focus on the case $r=0$.  Since $m=0$ we notice that the lemma is equivalent to 
 \begin{equation}
     \left\|\sum_n \sum_{p\neq n} \frac{a_p f_p}{p^2+\lambda-n^2}\right\|_{L^2}^2< +\infty.
 \end{equation}
 \textcolor{black}{Note that this} {\bf cannot be directly deduced from} the quadratically close inequality that we proved in Lemma \ref{lem-rbheat-1} (3) (case $m= s= 0$),
  \begin{equation}
     \sum_n \left\|\sum_{p\neq n} \frac{a_p f_p}{p^2+\lambda-n^2}\right\|_{L^2}^2< +\infty.
 \end{equation}
 
 Indeed, we need more delicate estimates.  {\color{black}We know from the fact that  $ \sum_{n\in\mathbb{N}^{*}} (q_n-\frac{a_n f_n}{\lambda})$ belongs to $H^{-1}_1$ that  
 \begin{equation}\label{H-1toL2}
     \sum_{n\in\mathbb{N}^{*}} (q_n-\frac{a_n f_n}{\lambda})= \sum_n \sum_{p\neq n} \frac{a_p f_p}{p^2+\lambda-n^2}= \sum_p \sum_{n\neq p} \frac{a_p f_p}{p^2+\lambda-n^2} \; \textrm{ in } H^{-1}_1,
 \end{equation}
 the last equality can be obtained from the ``distribution sense": we observe that the inner product of both quantities with $f_k$ are equivalent, which implies that those two quantities are equivalent.  Or, alternatively,  since
\begin{align*}
  &\;\;\;\;\sum_n \sum_{p\neq n}  \left\|\frac{a_p f_p}{p^2+\lambda-n^2}\right\|_{H^{-1}_1}\\
    & \leq C     \sum_n \sum_{p\neq n} \frac{1}{p |p^2-n^2|} \\
     & \leq C     \sum_n \left(\sum_{p< n}+ \sum_{n< p< 2n}+ \sum_{p> 2n}\right) \frac{1}{p |p^2-n^2|} \\
    & \leq C \sum_n  \left(\sum_{p< n} \frac{1}{n^2}(\frac{1}{p}+ \frac{1}{n-p})+ \sum_{n< p< 2n} \frac{1}{n^2} \frac{1}{p-n}+ \sum_{p> 2n}\frac{1}{p^3}\right)\\
    & \leq C \sum_n  \frac{1+\log n}{n^2} <+\infty, 
 \end{align*}
then thanks to Fubini we have that
 \begin{equation*}
    \sum_n \sum_{p\neq n} \frac{a_p f_p}{p^2+\lambda-n^2}= \sum_p \sum_{n\neq p} \frac{a_p f_p}{p^2+\lambda-n^2} \; \textrm{ in } H^{-1}_1.
 \end{equation*}
 }
 
Hence,  it suffices to  show the last quantity in equation \eqref{H-1toL2} belongs to $L^2_1$, which is equivalent to  
  \begin{equation}\label{ine-sum-l2}
     \left\|\sum_p a_pf_p \sum_{n\neq p} \frac{1}{p^2+\lambda-n^2}\right\|_{L^2_1}^2=  \sum_p a_p^2 \left(\sum_{n\neq p} \frac{1}{p^2+\lambda-n^2}\right)^2 < +\infty,
 \end{equation}
 which, to be combined with the growth condition $a_n\sim 1$, is equivalent to 
   \begin{equation}
   \sum_p  \left(\sum_{n\neq p} \frac{1}{p^2+\lambda-n^2}\right)^2 < +\infty.
 \end{equation}
 It further suffices to prove that 
  \begin{equation}
   \sum_p  \left(\sum_{n\neq p} \left|\frac{1}{p^2+\lambda-n^2}\right|\right)^2 < +\infty,
 \end{equation}
 which further \textcolor{black}{reduces to showing that}
   \begin{equation}
   \sum_{p\geq N}  \left(\left(\sum_{n< p}+ \sum_{p+1<n<2p}+ \sum_{n\geq 2p}\right) \left|\frac{1}{p^2+\lambda-n^2}\right|\right)^2=: \sum_{p\geq N} (S_p^1+ S_p^2+ S_p^3)^2 < +\infty.
 \end{equation}
 Next we estimate the $S_p^i$ one by one. 
 
 i) \textcolor{black}{Estimation of $S_p^{1}$:} \begin{align*}
     S_p^1&= \sum_{n\leq p-1} \frac{1}{p^2+\lambda-n^2}, \\
     &\leq \sum_{n\leq p-1} \frac{1}{p^2-n^2},\\
     &\leq \frac{1}{p} \sum_{n\leq p-1} \frac{1}{p-n}, \\
     &\leq \textcolor{black}{C}\frac{1+ \log p}{p},
 \end{align*}
 where $C$ is a constant independent of $p$.
 
 ii) \textcolor{black}{Estimation of $S_p^{2}$:} by the choice of $p\geq N$ and the fact that $\lambda=N$, we know that 
  \begin{align*}
     S_p^2&= \sum_{p<n<2p} \frac{1}{n^2-p^2-\lambda}, \\
     &\leq \sum_{p<n<2p} \frac{2}{n^2-p^2},\\
     &\leq \frac{1}{p} \sum_{p<n<2p} \frac{1}{n-p}, \\
     &\leq \textcolor{black}{C}\frac{1+ \log p}{p}.
 \end{align*}

 iii) \textcolor{black}{Estimation of $S_p^{3}$:} the last part can be estimated by 
  \begin{align*}
     S_p^3&= \sum_{2p\leq n} \frac{1}{n^2-p^2-\lambda}, \\
     &\leq \sum_{2p\leq n} \frac{2}{n^2-p^2},\\
     &\leq \sum_{2p\leq n} \frac{4}{n^2}, \\
     &\leq \frac{3}{p}.
 \end{align*}
 
 Consequently, 
 \begin{equation}
     \sum_{p\geq N} (S_p^1+ S_p^2+ S_p^3)^2\leq  \sum_{p\geq N} \textcolor{black}{C\left(\frac{1+ \log p}{p}\right)^2} < +\infty
 \end{equation}
 
Finally, for the case $r\in [0, 1/2)$, by slightly  modifying \eqref{ine-sum-l2}
  it suffices to show that 
\begin{equation}
    \sum_{p\geq N} p^{2r} (S_p^1+ S_p^2+ S_p^3)^2\leq \textcolor{black}{C}\sum_{p\geq N} p^{2r}\left(\textcolor{black}{\frac{1+ \log p}{p}}\right)^2  < +\infty,
\end{equation}
where $r<1/2$ is sharp. \textcolor{black}{This ends the proof of Lemma \ref{lem-rbheat-2}}.
 \end{proof}

\subsection{On the choice of the backstepping candidate}\label{sec-feedback-candidate}
{\color{black}
{\bf 
\textcolor{black}{I}n the following, in order to simplify the calculation,  we only  consider the special case $m=0$, thus
} 
$$c< |a_n|< C,$$ 
{\bf \textcolor{black}{the} other cases can be treated \textcolor{black}{exactly} similarly.}

According to Section \ref{sec-set-up-back} we know that every sequence $\{K_n\}_{n\in \N^*}$ determines a unique transformation $T$, and at least formally, it satisfies 
\begin{equation}\label{op-eq-ver}
    TA+ \phi K= AT-\lambda T.
\end{equation}
Thanks to Section \ref{sec-riesz-pro} we know that when $\{K_n\}_{\N^*}$ satisfies
\begin{equation}\label{kn-ver-dec}
    |K_n|\leq C, \; \forall n\in \N^*,
\end{equation}
this unique transformation $T$ is bounded on $H^s$ with $s\in (-3/2, 3/2)$. \\Moreover, if $\{K_n\}_{n\in \N^*}$ further satisfies 
\begin{equation}\label{kn-ver-dec-full}
    c\leq |K_n|\leq C, \; \forall n\in \N^*,
\end{equation}
then $T$ is an isomorphism on $H^s$ with $s\in (-3/2, 3/2)$.\\
However, until now we have \textcolor{black}{not yet treated} the second condition: \begin{equation}\label{tphi-eq-ver}
    T\phi= \phi.
\end{equation} 
\textcolor{black}{This condition is called in the literature the ``$TB=B$ condition" (here the function $\phi$ represent what is usually formally denoted $B$)}, which is now becoming a standard requirement for Fredholm type backstepping transformations, \cite{CGM,coron:hal-03161523, CoronLu14,CoronLu15,GLM, zhang:hal-01905098}.
The aim of this section is to determine \textcolor{black}{a} precise candidate of   $\{K_n\}_{n\in \N^*}$ such that
\begin{itemize}
    \item[(i)] The \textcolor{black}{``$TB=B$ condition"} \eqref{tphi-eq-ver} holds, \textcolor{black}{in a suitable space to be found.}
      \item[(ii)] The \textcolor{black}{boundedness} condition  \eqref{kn-ver-dec} holds.
     \item[(iii)]  The operator equality \eqref{op-eq-ver} holds, \textcolor{black}{in a suitable space to be found.}
\end{itemize}
The proofs of (i) and (ii) are provided by Sections \ref{sub:uniqueness}--\ref{subsec-Kn-bound}, and Section \ref{subsubsec-2}  is devoted to the proof of (iii).

We also remark here that, in this section, we only prove the condition \eqref{kn-ver-dec} instead of the stronger \eqref{kn-ver-dec-full}, as it is sufficient to obtain conditions \eqref{tphi-eq-ver} and \eqref{op-eq-ver}. The proof of \eqref{kn-ver-dec-full}, which implies that $T$ is an isomorphism and not only a bounded operator, is left to Sections \ref{sec-ine-kn-pro}--\ref{sec-T-inv-L2}.

}

\subsubsection{The $TB=B$ condition in weaker spaces}\label{sub:uniqueness}
For any given $\phi$  satisfying
\begin{equation}
     \phi=\sum_{n\in \N^*} a_n f_n, \textrm{ with }  c\leq |a_n|\leq C, \label{cond-an}
\end{equation}
we want to find $K$ such that
\begin{gather}
    K=\sum_{n\in \N^*} K_n f_n, \textrm{ with } c\leq|K_n|\leq C, \\
    T \phi= \phi, \textrm{ in a suitable sense}.\label{TB=B}
\end{gather}

{\color{black}Since the solutions of the closed-loop system are expected to live in $C^0([0, \tau]; L^2_1)$, it is natural to start by  considering that $T\phi=\phi$ holds in $L^2_1$.} 

For any given $\{a_n\}_{n\in \N^*}$ satisfying \eqref{cond-an}, Lemma \ref{lem-rbheat-1} (2) holds. Meanwhile the strong  $T\phi=\phi$ condition for \eqref{TB=B} reads  
\begin{equation}
    \sum_p a_p f_p=\phi= T\phi= \sum_n a_n h_n= \sum_n -a_n K_n q_n,
\end{equation}
which is equivalent to
\begin{equation}\label{eq-tbb}
   \phi=  -\sum_n a_n K_n q_n.
\end{equation}
However, we can not use   Lemma \ref{lem-rbheat-1} \textcolor{black}{(3) with ${\color{ForestGreen}m=}s=0$} directly, because $\phi\in H^{-\frac{1}{2}-}_1$ instead of $L^2_1$. Otherwise, formally the preceding equation admits a solution. \\

\color{black}We now prove that the condition \eqref{tphi-eq-ver} can hold in $H^{-1/2-}_1$.
We notice that, formally,
\begin{equation}\label{eq-an-kn-Riesz}
     T\phi=\sum_n a_n K_n q_n=    \sum_n \frac{a_n K_n}{n} (n q_n).
\end{equation}
As $\{n q_n\}_{n\in \N^*}$ is a Riesz basis of $H^{-1}_1$, and as $\phi\in H^{-1}_1$, \eqref{eq-tbb} can be solved in $H^{-1}_1$. Indeed, since the span of $\{n q_n\}_{n\in \N^*}$ is dense in $H^{-1}_1$, and since $\phi$ also belongs to $H^{-1}_1$, it has the unique decomposition:
\begin{equation}
    \phi=\sum_{n\in \N^\ast} \tilde{a}_n (nq_n), \quad (\tilde{a}_n ) \in \ell^2(\N^\ast).
\end{equation}
We then define 
\begin{equation}\label{1st-def-Kn}
    K_n:=-n \frac{\tilde{a}_n}{a_n}, \quad n \in \N^\ast.
\end{equation}
Note that $(K_n)$ is thus uniquely determined, solves \eqref{eq-tbb}, and
\begin{equation}\label{Kn-growth-an}
    \left(\frac{a_n K_n}{n}\right)_{n\in \N^\ast} \in \ell^2(\N^\ast).
\end{equation}

Now, as $\phi\in H^{-1/2-}_1$, by Lemma \ref{lem-rbheat-1} (3), and uniqueness of $\tilde{a}_n$, we have
\begin{equation}\label{phi-decomp-H-s}
    \phi=\sum_{n\in \N^\ast} \tilde{a}_n n^{s+1}(n^{-s}q_n), \quad s\in (-3/2, -1/2),
\end{equation}
with 
\[(\tilde{a}_n n^{s+1})_{n\in \N^\ast}\in \ell^2(\N^\ast), \quad s\in (-3/2, -1/2).\]
Together with \eqref{1st-def-Kn}, this implies
\begin{equation}\label{Kn-growth-an-s}
    (a_nK_n n^s)_{n\in \N^\ast} \in \ell^2(\N^\ast), \quad s\in (-3/2, -1/2),
\end{equation}
which generalizes \eqref{Kn-growth-an}.
We can now write, analogously to \eqref{eq-an-kn-Riesz}:
\begin{equation}\label{eq-an-Kn-Riesz-s}
     T\phi=\sum_n a_n K_n q_n=    \sum_n a_n K_n n^s (n^{-s} q_n), \quad s\in (-3/2, -1/2).
\end{equation}
Together with \eqref{phi-decomp-H-s}, this solves \eqref{eq-tbb}, i.e. \eqref{tphi-eq-ver}, in $H^s_1$ for $s\in (-3/2, -1/2)$

Note however that, although condition \eqref{Kn-growth-an-s} holds for any bounded sequence $\{K_n\}_{n\in \N^*}$, conversely \eqref{Kn-growth-an-s} is not enough to conclude that $\{K_n\}_{n\in \N^*}$ is bounded.  Indeed, let $b_n:= \log n$, we can easily observe that
\begin{equation}
    \left\{\frac{b_n}{n^{1/2+\varepsilon}}\right\}_{n\in \N^*}\in\ell^2(\N^{*}), \; \forall \varepsilon>0.
 \end{equation}

\begin{remark}
More generally, even if a sequence $\{b_n\}_{n\in \N^*}$ satisfies 
\begin{equation}\label{verif-l2-notin}
    \left\{\frac{b_n}{n^{\varepsilon}}\right\}_{n\in \N^*}\in\ell^2(\N^{*}), \; \forall \varepsilon>0,
 \end{equation}
 we are not able to conclude that the sequence $\{b_n\}_{n\in \N^*}$ is bounded. For example, by defining 
 \begin{gather}
     M(n)= [e^n]+ 1, \forall n\in \N^*, \\
     b_m:= n, \textrm{ if } m= M(n), \\
     b_m:= 0, \textrm{ if } m\notin \{M(n): n\in \N^*\},
 \end{gather}
 we know that for any $\varepsilon>0$, 
 \begin{equation}
     \sum_{m\in \N^*} \left(\frac{b_m}{m^{\varepsilon}}\right)^2=   \sum_{m\textcolor{black}{\in \{M(n): n\in \N^*\}}} \left(\frac{b_m}{m^{\varepsilon}}\right)^2\leq \sum_{n\in \N^*} \left(\frac{n}{e^{\varepsilon n}}\right)^2<+\infty.
 \end{equation}
 Clearly, such  $\{b_n\}_{n\in \N^*}$ is not bounded.
\end{remark}

In the next section, we will prove that the sequence   $\{K_n\}_{n\in \N^*}$  that we have found in this section is indeed bounded from above. \color{black}

\subsubsection{Regularity of $K$ and $T$}\label{subsec-Kn-bound}

{\color{black}In order to prove the condition \eqref{kn-ver-dec}, \textcolor{black}{namely the boundedness of the candidate $\{K_{n}\}_{n\in\mathbb{N}^{*}}$,} we come back to  the $T\phi=\phi$ equation }
 \begin{equation}\label{eq-knan}
     \sum_n -a_n K_n q_n= \sum_n a_n f_n.
 \end{equation}
 {\color{black}Thanks to the preceding Section, we have found \textcolor{black}{a} unique  $\{K_n\}_{n\in \N^*}$ satisfying \eqref{Kn-growth-an-s} such that equation \eqref{eq-knan} holds in $H^{-1/2-}_1$  (in particular, in  $H^{-1}_1$).}

 Motivated \textcolor{black}{by} Lemma \ref{lem-rbheat-2} part (3), we define $c_n$ as
 \begin{equation}\label{def-cn-kn}
     c_n =: -a_n K_n-\lambda.
 \end{equation}
  
{\color{black} 
As $\{n q_n\}_{n\in \N^*}$ is a Riesz basis of $H^{-1}_1$, we know that  
\begin{equation}
    \sum_n \lambda q_n=   \lambda \sum_n \frac{1}{n} (n q_n)\in H^{-1}_1.
\end{equation}
Then, the $T\phi=\phi$ condition in $H^{-1}_1$   is equivalent to 
 \begin{equation}\label{eq:cnqnh11}
     \sum_n (\lambda+ c_n)q_n= \sum_n a_n f_n \; \textrm{ in } H^{-1}_1,
 \end{equation}
and 
 \begin{equation}\label{eq:preceding}
     \sum_n  c_n q_n= \sum_n (a_n f_n- \lambda q_n) \; \textrm{ in } H^{-1}_1.
 \end{equation} }
Thanks to Lemma \ref{lem-rbheat-2} part (3) with $r=0$, we actually have
  \begin{equation}\label{eq:cnqnl21}
      \sum_n  c_n q_n= \sum_n (a_n f_n- \lambda q_n)\in L^2_1.
  \end{equation}
 Considering the fact that $\{q_n\}$ is a Riesz basis of $L^2_1$ (Lemma \ref{lem-rbheat-1} (3) with $m=s=0$), by uniqueness we get
 \begin{equation}\label{eq:c_n}
     \{c_n\}_{n\in \N^*}\in\ell^2(\N^{*})
 \end{equation}
 
Now, we have by definition of $c_n$:
  \begin{equation}
      K_n= -\frac{\lambda+ c_n}{a_n}.
  \end{equation}
  By \eqref{cond-an} and \eqref{eq:c_n}, $(K_n)$ is clearly bounded.

  \textcolor{black}{Now, from} Lemma \ref{lem-rbheat-1} (4) in the case $m=0$, we deduce that $T$ is  bounded on $H^s_{\textcolor{black}{1}}$ for $s\in (-3/2, 3/2)$.

 \begin{remark}
 We notice that the equation  \eqref{eq-knan} is equivalent to 
 \begin{equation}
     \sum_{n\in \N^*} (a_n K_n) \frac{1}{p^2+\lambda-n^2}= -1, \; \forall p\in \N^*.
 \end{equation}
 Hence the value of $\{-a_n K_n\}_{n\in \N^*}= \{\lambda+ c_n\}_{n\in \N^*}$ is independent of the choice of $\{a_n\}_{n\in \N^*}$ satisfying
 \[c\leq |a_n|\leq C.\]
 \end{remark}

\subsubsection{On the operator equality}\label{subsubsec-2}
\textcolor{black}{In the previous subsection we proved the following lemma} 
 \begin{lem}\label{lem-tb=b} By the  choice of $\{a_n\}_{n\in \N^*}$ and $\{K_n\}_{n\in \N^*}$ from Section \ref{subsec-Kn-bound}, we have
 \begin{equation}
     T\phi= \phi \textrm{ in } H^{s}_1, \quad s\in (-3/2, 1/2),
 \end{equation}
 in particular it holds in $H^{-1}_1$. However, it does not hold in $L^2_1$.
 \end{lem}
Now, we show the operator equality and the spaces in which this equality holds.
 \begin{lem}\label{lem-op-eq-tbk}
 Let $r\in (-1/2, 1/2)$. By the  choice of $\{a_n\}_{n\in \N^*}$ and $\{K_n\}_{n\in \N^*}$ from Section \ref{subsec-Kn-bound}, for any $\varphi\in H^{r+1}_1$ we have
 \begin{equation}\label{eq:op-eq-2}
     (TA+ T\phi K)\varphi= (AT-\lambda T)\varphi \textrm{ in } H^{r-1}_1.
 \end{equation}
 In particular we can consider $r=0$ then the equality holds in $H^{-1}_1$. \textcolor{black}{Moreover, the range} $r\in (-1/2, 1/2)$ is sharp. 
 \end{lem}
 \begin{proof}  Recalling that, thanks to Section \ref{subsec-Kn-bound}, $T:H^s_1\rightarrow H^s_1$ is a bounded operator for $s\in (-3/2, 3/2)$.  Though so far we do not know whether $T$ is an isomorphism --this will be proved in the next section-- \textcolor{black}{this} is now sufficient to study the operator equality.

Let us first check that each {\color{ForestGreen}term} of the operator equality \eqref{eq:op-eq-2} is well defined.
 \begin{itemize}
     \item[i)]  $TA\varphi$: we know that $A\varphi\in H^{r-1}_1$. Notice that $r-1\in (-3/2, -1/2)$, which, combined with Lemma \ref{lem-rbheat-1} (4) in the case $m=0$ and $s:=r-1$, \textcolor{black}{implies} that $TA\varphi\in H^{r-1}_1$. Moreover, as $A: H^{r+1}_1\rightarrow H^{r-1}_1$ and $T: H^{r-1}_1\rightarrow H^{r-1}_1$ are bounded, the linear operator  $TA: H^{r+1}_1\rightarrow H^{r-1}_1$ is bounded.\\
     {\color{black} Concerning the sharpness of $r\in (-1/2, 1/2)$, as we can see above, if $r$ is chosen such that $r\leq -1/2$ then the operator $T$ is no longer bounded on $H^{r-1}_1$, thus the linear operator  $TA: H^{r+1}_1\rightarrow H^{r-1}_1$ is not bounded. }
     
     \item[ii)] $T\phi K\varphi$: \textcolor{black}{given that} $\varphi\in H^{r+1}_1$ with $r>-1/2$, and \textcolor{black}{that $K: H^{r+1}_1\rightarrow \R$ is a bounded operator,} we know that $|K\varphi|<+\infty$ is well-defined. Then, thanks to Lemma \ref{lem-tb=b}, $T\phi K\varphi\in H^{-1/2-}_1\subset H^{r-1}_1$. 
     Moreover, as  and $T\phi: \R \rightarrow H^{r-1}_1$ \textcolor{black}{and $K$} are bounded, the linear operator  $T\phi K: H^{r+1}_1\rightarrow H^{r-1}_1$ is bounded.
     
     \item[iii)]  $T\varphi$: thanks to Lemma \ref{lem-rbheat-1} \textcolor{black}{(4) in} the case $m=0$ and $s=r+1\in (1/2, 3/2)$, and the fact that $\varphi\in H^{r+1}_1$, we know that $T\varphi\in H^{r+1}_1\subset H^{r-1}_1$. Clearly, $T: H^{r+1}_1\rightarrow H^{r-1}_1$ is bounded (in fact $T$ is even bounded from $H^{r+1}_{1}$ in itself).
     
     \item[iv)]  $AT\varphi$: because $T\varphi\in H^{r+1}_1$, we have that  $AT\varphi\in H^{r-1}_1$. As $T: H^{r+1}_1\rightarrow H^{r+1}_1$ and $A: H^{r+1}_1\rightarrow H^{r-1}_1$  are bounded, the linear operator  $AT: H^{r+1}_1\rightarrow H^{r-1}_1$ is bounded.\\
      {\color{black} Concerning the sharpness of $r\in (-1/2, 1/2)$, as we can see above, if $r$ is chosen such that $r\geq 1/2$ then the operator $T$ is no longer bounded on $H^{r+1}_1$, thus the linear operator  $AT: H^{r+1}_1\rightarrow H^{r-1}_1$ is not bounded. }
 \end{itemize}

 {\color{black}From the above we have:
 \[TA,\; T\phi K,\; AT,\; \lambda T: H^{r+1}_1\rightarrow H^{r-1}_1,\]
 are  bounded linear operators.  Now, using the definition of $q_{n}$ given in Section \ref{sec-set-up-back}, we have in the $H^{r-1}$ space:
 \begin{align*}
     (TA+T\phi K-AT+\lambda T) f_n&= (-A-n^2+\lambda) Tf_n+ \phi  K_n,\\
     &= K_n\left(\phi - (-A-n^2+\lambda)q_n\right), \\
     &= 0, \\
     i.e. \quad TAf_n+T\phi Kf_n &= ATf_n-\lambda T f_n.
 \end{align*}
 
 Since the finite linear combinations of $\{f_n\}_{n\in \N^*}$ \textcolor{black}{are} dense in $H^{r+1}_1$, by boundedness of $TA, AT, T\phi K$ and $T$, this yields \eqref{eq:op-eq-2}.}
 \end{proof}
 
 \subsection{Invertibility of the transformation $T$ on the space $H^{-1}_1$.}\label{sec-ine-kn-pro}
So far we know, thanks to Lemma \ref{lem-rbheat-1}, that $T$ is a bounded operator on $H^{s}_1$. But we do not know yet that it is an isomorphism. To prove this, we will first prove in this section that $T$ is invertible on $H^{-1}_1$. Then, we will show in Section \ref{sec-T-inv-L2} that this implies that it is invertible also on $H^{s}_1$ for $s\in(-3/2,3/2)$.\\

The key lemma is the following:
 
 \begin{lem}\label{lem-key-compact}
 Let $r\in [0, 1/2)$. The operator 
 \[\tilde{T}:= T-Id: L^2_1\rightarrow H^r_1,\]
  \[\textcolor{black}{(\text{resp. }}\tilde{T}:= T-Id: H^{-1}_1\rightarrow H^{-1+r}_1\textcolor{black}{)},\]
 \textcolor{black}{is a} continuous operator. Hence $\tilde{T}:L^2_1\rightarrow L^2_1$  (resp. $T:H^{-1}_1\rightarrow H^{-1}_1$), and $T$ is a Fredholm operator on $L^2_1$ (resp. on $H^{-1}_1$).
 \end{lem}
 
To prove this Lemma, we need the following result: 
 
\begin{lem}\label{lem-qn-fn-l2-hr-bn} Let $r\in [0, 1/2)$. \textcolor{black}{There exists a constant $C>0$ such that}
 \begin{equation}
     \left\|\sum_n b_n (q_n-\frac{a_n f_n}{\lambda})\right\| _{H^r_1}^2\leq C \sum_{n} b_n^2, \quad \forall (b_n)\in \ell^2(\N).
       \end{equation}
 \end{lem}
 \begin{proof}[Proof of Lemma \ref{lem-qn-fn-l2-hr-bn}]
  In fact, we have \textcolor{black}{from the definition of $\{q_{n}\}_{n\in\mathbb{N}^{*}}$, the Riesz basis property of $\{f_{p}\}_{p\in\mathbb{N}^{*} }$ and the growth condition \eqref{cond-an} for $\{a_{p}\}_{p\in\mathbb{N}^{*}}$:}
 \begin{align*}
      \left\|\sum_n b_n (q_n-\frac{a_n f_n}{\lambda})\right\| _{H^r_1}^2&=   \left\|\sum_p a_pf_p \sum_{n\neq p} \frac{b_n}{p^2+\lambda-n^2}\right\|_{H^r}^2,\\
      &\leq C \sum_p p^{2r} \left(\sum_{n\neq p} \frac{b_n}{p^2+\lambda-n^2} \right)^2,\\
      &\leq C \sum_p p^{2r} \left(\sum_{n\neq p} b_n^2\right)  \left(\sum_{n\neq p} \frac{1}{(p^2+\lambda-n^2)^2} \right), \\
      &\leq C \left(\sum_{n} b_n^2\right) \sum_p p^{2r}  \left(\sum_{n\neq p} \frac{1}{(p^2+\lambda-n^2)^2} \right), \\
      &\leq C \sum_{n} b_n^2 ,
 \end{align*}
 where in the last step we have used Lemma \ref{lem-rbheat-2} (1) (with $m=0$ and $a_n=1$).
  \end{proof}
 
 \begin{proof}[Proof of Lemma \ref{lem-key-compact}]
i) We first consider the $L^2_1$ case. We notice that
 \[\tilde{T} f_n= -K_n q_n- f_n= \frac{\lambda+ c_n}{a_n}q_n- f_n= \left(\frac{\lambda}{a_n} q_n- f_n\right)+ \frac{c_n}{a_n} q_n.\] 
 
Therefore, for any $f=\sum_n b_n f_n\in L^{2}_1$  satisfying $\{b_n\}_{n\in \N^*}\in\ell^2( \N^*)$, we have that 
\begin{align*}
    \left\|\tilde{T} \sum_n b_n f_n\right\|_{H^r_1}^2=  \left\|\sum_n  \frac{b_n}{a_n}(\lambda q_n- a_n f_n)+  \frac{b_n}{a_n} c_n q_n\right\|_{H^r_1}^2.
\end{align*}
On the one hand, according to Lemma \ref{lem-rbheat-1} (5) the case that $m=0$, we know that 
\begin{align*}
    \left\|\sum_n  \frac{b_n}{a_n}(\lambda q_n - a_n f_n)\right\|_{H^r_1}^2\leq C \left(\sum_n b_n^2\right)\left(\sum_n \|\lambda q_n -f_n\|_{H^r_1}^2\right)\leq C  \sum_n b_n^2=C \|f\|^2_{L^2_1}.
\end{align*}
On the other hand, 
according to Lemma \ref{lem-rbheat-1} (3) with $m=0$ and $s=r$, we know that $\{n^{-r} q_n\}$ is a Riesz basis of $H^r_1$. Thus 
\begin{align*}
    \left\|\sum_n  \frac{b_n}{a_n} c_n q_n\right\|_{H^r_1}^2&=  \left\|\sum_n  (b_n a_n^{-1} c_n n^r) (n^{-r} q_n)\right\|_{H^r_1}^2 \\
    &\leq C \sum_n (b_n a_n^{-1} c_n n^r)^2 \\
     &\leq C \sum_n (b_n)^2 (c_n n^r)^2 \\
    &\leq C \sum_n b_n^2 \\
    & = C \|f\|^2_{L^2_1},
\end{align*}
where we have used the fact that $\{c_n n^r\}_{n\in \N^*}\in\ell^2(\N^{*}) \subset l^{\infty}(\N^{*})$ which is proved in Section \ref{subsec-Kn-bound}, equation \eqref{eq:c_n}. \textcolor{black}{Hence, the operator $\tilde T: L^{2}_{1}\rightarrow H^{r}_{1}$ is continuous}. By compactness of the injection $H^{r}_{1}\rightarrow L^2_1$,  $\tilde T: L^{2}_{1}\rightarrow L^2_{1}$ is compact.\\

ii) Next we prove the $H^{-1}_1$ case. For any 
$$f=\sum_n b_n f_n= \sum_n (b_n/n) (n f_n)\in H^{-1}_1,$$
we know that    $\{b_n/n\}_{n\in \N^*}\in\ell^2(\N^*)$, thus
\begin{align*}
    \left\|\tilde{T} \sum_n b_n f_n\right\|_{H^{-1+r}_1}^2=  \left\|\sum_n  \frac{b_n}{a_n}(\lambda q_n- a_n f_n)+  \frac{b_n}{a_n} c_n q_n\right\|_{H^{-1+r}_1}^2.
\end{align*}
On the one hand, according to Lemma \ref{lem-rbheat-2} (2) the case that $m=0$, we know that 
\begin{align*}
     \Big\|\sum_n  \frac{b_n}{n a_n}(n(\lambda q_n - a_n f_n)  & \Big\|_{H^{-1+r}_1}^2 \\
     &\leq C \left(\sum_n \left(\frac{b_n}{n}\right)^2\right)\left(\sum_n \|n(\lambda q_n -a_n f_n)\|_{H^{-1+r}_1}^2\right)\\
    &\leq C  \sum_n \left(\frac{b_n}{n}\right)^2 \\
    &\leq C \|f\|_{H^{-1}_1}^2.
\end{align*}
On the other hand, 
according to Lemma \ref{lem-rbheat-1} (3) the case that $m=0$ and $s=-1+r$, we know that $\{n^{1-r} q_n\}$ is a Riesz basis of $H^{-1+r}_1$. Thus 
\begin{align*}
    \left\|\sum_n  b_n  a_n^{-1} c_n q_n\right\|_{H^{-1+r}_1}^2&=  \left\|\sum_n  (b_n a_n^{-1} c_n n^{r-1}) (n^{1-r} q_n)\right\|_{H^{-1+r}_1}^2, \\
    &\leq C \sum_n \left(\frac{b_n}{n} a_n^{-1} c_n n^{r}\right)^2, \\
    &\leq C \sum_n \left(\frac{b_n}{n}\right)^2 \\
    &\leq C \|f\|_{H^{-1}_1}^2,
\end{align*}
where, again,  we have used the fact that $\{c_n n^r\}_{n\in \N^*}\in\ell^2(\N^{*}) \subset l^{\infty}(\N^{*})$. 

Therefore, for $r\in (0, 1/2)$, $T:H^{-1}_1\rightarrow H^{-1+r}_1$ is continuous. Since the inclusion $H^{-1+r}_1\rightarrow H^{-1}_1$ is compact, this concludes the proof of Lemma \textcolor{black}{\ref{lem-key-compact}}.

 \end{proof}
 
\textcolor{black}{We now prove the following invertibility result on $H^{-1}_{1}$.
\begin{lem}\label{lem:H-1}
$T :H^{-1}_{1}\rightarrow H^{-1}_{1}$ is an isomorphism.
\end{lem}
}

\begin{proof}[Proof of Lemma \ref{lem:H-1}]
Since $T: H^{-1}_1\rightarrow H^{-1}_1$ is a Fredholm operator of index 0 \textcolor{black}{(i.e. can be written as the sum of the identity and a compact operator)}, thanks to Lemma \ref{lem-key-compact}, it suffices to show that $\ker T^*=\{0\}$ to conclude. {\color{black} By definition of Fredholm operators, we know that both $\ker T$ and coKer $T$ have finite dimension. 
Moreover, since the adjoint of a Fredholm operator is still a Fredholm operator (by Schauder's theorem, the adjoint of a compact operator is still compact), $\ker T^*$ is of finite dimension.
More precisely, we have
\begin{equation*}
    \textrm{ dim ker }T= \textrm{ dim coKer }T= \textrm{ dim ker }T^*= \textrm{ dim coKer }T^* < +\infty. 
\end{equation*}
}  In the following we mimic the method of \textcolor{black}{\cite{CGM}} to show that for the operator $T: H^{-1}_1\rightarrow H^{-1}_1$ we have that  $\ker T^*= \{0\}$.
   The proof is divided in three steps
   \begin{itemize}
       \item[1)] There exists $\rho\in \mathbb{C}$ such that 
       \begin{gather*}
           A+ \textcolor{black}{\phi_{1}}K+\lambda \textcolor{black}{Id} +\rho \textcolor{black}{Id}: H^1_1\rightarrow H^{-1}_1 \textrm{ is invertible}, \\
            A+ \rho \textcolor{black}{Id}: H^1_1\rightarrow H^{-1}_1 \textrm{ is invertible}.
       \end{gather*}
         \item[2)] {\color{black} For such a complex number $\rho$,} $\ker T^*$ is stable under $(A+\rho \textcolor{black}{Id})^{-1}$. As it is finite-dimensional, \textcolor{black}{ $(A+\rho \textcolor{black}{Id})^{-1}$ has an eigenvector in $\ker T^*$}, $i.e.$ there exists $h\in \ker T^\ast$ and $\mu\neq 0$ such that $(A+\rho \textcolor{black}{Id})^{-1}h= \mu h$. Hence $h= f_k$ for some $k\in \N^*$.
          \item[3)] By adapting the $T\textcolor{black}{\phi_{1}}=\textcolor{black}{\phi_{1}}$ condition\textcolor{black}{, we} show that  $h= f_k$ is not in $\ker T^*$.
   \end{itemize}

   {\bf 1)}
   As the spectrum of  $A+\rho \textcolor{black}{Id}$ is explicit and discrete, in the following we mainly focus on the study of the spectrum of  $A+\phi K+\lambda \textcolor{black}{Id}+ \rho \textcolor{black}{Id}$. {\color{black} By denoting $z:= \lambda+ \rho$}, we try to investigate the invertibility of $\textcolor{black}{Id}+A^{-1}\phi K+z A^{-1}$ \textcolor{black}{in the $H^1_1$ space, with {\color{black} some} $z\in \mathbb{C}$}, {\color{black}which further gives some $\rho= z-\lambda$ such that the investigated two operators are invertible.}  Let us remark here that as working with $H_1^k$ spaces, $A$ is automatically invertible in this space, while when dealing with $H_2^k$ spaces, $A$ has a trivial kernel $1$, in this case we will replace $A$ by $\tilde A:= A+ \delta$ with $\delta\neq 0$ sufficiently close to 0. For such a purpose we  consider two cases. 
   
   i) If $K(A^{-1}\phi )\neq -1$, then we \textcolor{black}{know} that the bounded operator $\textcolor{black}{Id}+A^{-1}\phi K$ is invertible. In fact, for any $f\in H^1_1$, we can check that 
   \begin{equation}
       \varphi:= f-\frac{A^{-1}\phi  (Kf)}{1+ K(A^{-1}\phi )}\in H^1_1,
   \end{equation}
   solves
   \begin{equation}
       (\textcolor{black}{Id}+A^{-1}\phi K) \varphi= f.
   \end{equation}
   Since $\textcolor{black}{Id}+A^{-1}\phi K$ is invertible, and \textcolor{black}{since} $A^{-1}$ is also a bounded operator (indeed even compact) on $H^1_1$, according to the openness  of invertible operator, there exists a small ball $B_{\varepsilon}(0)$ around 0, such that 
   \begin{equation}
       (\textcolor{black}{Id}+A^{-1}\phi K)+ z A^{-1} \textrm{ is  invertible in } H^1_1, \;  \forall z\in  B_{\varepsilon}(0).
   \end{equation}
   
   ii) If $K(A^{-1}\phi )= -1$, then we can easily check that 0 is en eigenvalue of $\textcolor{black}{Id}+A^{-1}\phi K$ with multiplicity 1 and the eigenspace \textcolor{black}{is} generated by $A^{-1}\phi $.  
   
   According to the perturbation theory, see for example \cite{MR47254}, there exist  small open neighborhoods $\Omega$ and  $\widetilde{\Omega}$ of $0\in \mathbb{C}$ satisfying
   \begin{gather*}
       (\textcolor{black}{Id}+ A^{-1}\phi  K+ zA^{-1}) y(z)= \lambda(z) y(z), \\
       y(z): z\in \Omega\mapsto y(z)\in H^1_1 \textrm{ is holomorphic,} \\
        \lambda(z): z\in \Omega\mapsto \lambda(z)\in \widetilde{\Omega}\subset \mathbb{C} \textrm{ is holomorphic,}\\
        \lambda(0)=0,\quad y_0:= y(0)= A^{-1}\phi ,
   \end{gather*}
in such a way that {\color{black} for any $z\in \Omega$}, $\lambda(z)$ is the unique eigenvalue inside $\widetilde{\Omega}$. Recall that $\lambda(0)=0$ and that the zero points of any non-trivial (not identically zero) holomorphic fucntion are isolated.

   If further there exists a small neighborhood $\omega$ of $0\in \mathbb{C}$ such that $\lambda(z)=0$ for any $z\in \omega$, then
   we are able to decompose $y(z)\in H^1_1$ is power series as
   \begin{equation}
       y(z)= y_0+ \sum_{k=1}^{+\infty} z^k y_k, \; \; y_k\in H^1_1.
   \end{equation}
   
   By matching the coefficients of the power series, we get 
   \begin{equation}
       (\textcolor{black}{Id}+ y_0 K)y_k+ A^{-1} y_{k-1}=0, \; \forall k\geq 1.
   \end{equation}
  By adapting $A^{-1}$ and $K$ to the preceding equation we  conclude that 
  \begin{equation}
      K(A^{-1}y_k)=0, \; \forall k\geq 0.
  \end{equation}
  Then by successively adapting $A^{-1}$ and $K$ to the same equation we arrive at 
   \begin{equation}
      K(A^{-n}y_k)=0, \; \forall k\geq 0, \forall n\geq 1,
  \end{equation}
  which in particular yields 
  \begin{equation}
      K(A^{-n} y_0)=0,  \forall n\geq 1.
  \end{equation}
  
  The preceding equality implies
  \begin{equation}
      \sum_n \frac{a_n K_n}{n^{2l}}=0, \forall l\geq 2.
  \end{equation}
  Again using the holomorphic function technique, we conclude that $a_n K_n=0$, which is a contradiction.
  
   Therefore, there exists a sequence of $\{z_k\}$ converging to 0 such that $\lambda(z_k)\neq 0$. Indeed, thanks to  the fact that $\lambda$ is holomorphic with $0$ being a zero point, in this case we even have that $\lambda(z)\neq 0$ in $\omega_1\setminus\{0\}$ with $\omega_1$ being a small neighborhood of $0$. Then, since $\lambda(z_k)\neq 0$ and $\lambda(z_k)$ is the unique eigenvalue inside $\widetilde{\Omega}$,  for $z_k$ sufficiently close to 0 we know that $\textcolor{black}{Id}+A^{-1}\phi K+ z_k A^{-1} $ is invertible. {\color{black} Thus ${Id}+A^{-1}\phi K+ \lambda A^{-1}+ (z_k-\lambda) A^{-1}$ is invertible. } As the spectrum of $A+ \rho \textcolor{black}{Id}$ is discrete, we can find a $\rho$ (more precisely, some $z_k- \lambda$), such that both $A+ \rho \textcolor{black}{Id} $ and $A+ \phi K+ \lambda \textcolor{black}{Id}+ \rho \textcolor{black}{Id}$ are invertible.   \\

    {\bf 2)} Because 
    \begin{equation}
        T(A+\phi K+\lambda \textcolor{black}{Id}+ \rho \textcolor{black}{Id})= AT+ \rho \textcolor{black}{T}: H^1_1\rightarrow H^{-1}_1,
    \end{equation}
    we have 
    \begin{equation}
        (A+ \rho \textcolor{black}{Id})^{-1} T= T(A+\phi K+\lambda \textcolor{black}{Id}+\rho \textcolor{black}{Id})^{-1}: H^{-1}_1\rightarrow H^{1}_1.
    \end{equation}
    
    Suppose that  $h\in \ker T^*$, then we deduce from the above operator equality that 
    \begin{align*}
        0&= \langle (A+\rho \textcolor{black}{Id})^{-1}T \varphi- T(A+\phi K+\lambda \textcolor{black}{Id}+\rho \textcolor{black}{Id})^{-1}\varphi, h\rangle_{H^{-1}_1}, \\
        &= \langle  \varphi, T^* (A+\rho \textcolor{black}{Id})^{-1} h\rangle_{H^{-1}_1}- \langle  \textcolor{black}{(A+\phi K+\lambda \textcolor{black}{Id}+\rho \textcolor{black}{Id})^{-1}}\varphi, T^*  h\rangle_{H^{-1}_1}, \\
         &= \langle  \varphi, T^* (A+\rho \textcolor{black}{Id})^{-1} h\rangle_{H^{-1}_1}, \quad \forall \varphi\in H^{-1}_1,
    \end{align*}
    where we have used the fact that $A$ is self-adjoint. 
    
   The above implies $(A+\rho \textcolor{black}{Id})^{-1} h\in$ $\ker T^*$, thus
    \begin{equation}
        (A+ \rho {Id})^{-1}: \ker T^*\rightarrow \ker T^*.
    \end{equation}
    
   Suppose that $\ker T^*$ is not reduced to $\{0\}$. Therefore,  because the  space $\ker T^*$ is of finite dimension \textcolor{black}{and not reduced to $\{0\}$} we can find an eigenfunction $(h, \mu)$, $h\neq 0$ and $\mu\neq 0$, such that \begin{equation}
       (A+\rho \textcolor{black}{Id})^{-1}h= \mu h \textcolor{black}{\textrm{ and } h\in \ker T^{*}}.
   \end{equation} 
   
    We immediately deduce that $h$ is an eigenfunction of $A=\Delta$ in $H^{-1}_1$. Moreover, we know from the definition of $h$ that
    \[h=\mu (A+\rho {Id}) h, \quad i.e. \quad A h= \frac{1-\rho}{\mu} h,\]
    and $h$ is a eigenfunction of $A$. \textcolor{black}{We notice that the subspaces of $H^{1}_{1}$ that are also eigenspaces of $\Delta$ have dimension 1 (the eigenvalues are not degenerate in $H^1_{1}$).} In particular the dimension of the eigenspace of $(1-\rho)/\mu$ is one, \textcolor{black}{and therefore} $h= C f_k$ for some $k\in \N^*$.
    \\
    
     {\bf 3)} From the above, we get
     \begin{equation}
         \langle T\varphi, f_k \rangle_{H^{-1}_1}=0, \; \forall \varphi\in H^{-1}_1. 
     \end{equation}
     Thanks to the fact that  $T\phi=\phi$ in $H^{-1}_1$, we can take $\varphi:= \phi$ to achieve
      \begin{equation}
         0=\langle T\phi, f_k \rangle_{H^{-1}_1}= \langle \phi, f_k \rangle_{H^{-1}_1}= \frac{a_k}{k^2}, 
     \end{equation}
     which is in contradiction with the fact that $a_n\neq 0$.
   \end{proof}

\subsection{Invertibility of the transformation $T$ on $H^s_1$ for any $s\in (-3/2, 3/2)$}\label{sec-T-inv-L2}

In the previous part we have proved the invertibility of the transformation $T$ on $H^{-1}$. \textcolor{black}{We will now show that this implies 
\begin{equation*}
    K_{n}\neq 0,\;\forall n\in\mathbb{N}^{*},
\end{equation*}
which will in turn imply the isomorphism property in $H^{s}_{1}$ for any $s\in(-3/2,3/2)$ thanks to Lemma \ref{lem-rbheat-1}.}
 
 Recall \textcolor{black}{that} 
 \begin{equation}
    T: n f_n\mapsto - K_n (n q_n),
\end{equation}
is a bounded operator form $H^{-1}_1$ to itself. By contradiction, suppose that for some $n_0$ we have $K_{n_0}=0$. \textcolor{black}{As $T:H_{1}^{-1}\rightarrow H_{1}^{-1}$ is an isomorphism there exists} $h\in H^{-1}_1$ such that
\[Th= n_0 q_{n_0}.\]
\textcolor{black}{By property of Riesz bases, there exists a unique $\{d_n\}_{n\in \N^*}\in\ell^2(\N^{*})$ such that}
\begin{equation}
    h=\sum_n d_n (n f_n)\in H^{-1}_1.
\end{equation}
Then 
\begin{equation}
   Th= \sum_n - d_n K_n (n q_n)
\end{equation}
which converges absolutely as $\{n q_n\}_{n\in \N^*}$ is a Riesz basis of $H^{-1}_1$. By definition of $h$, and by uniqueness of the decomposition, 
\[d_{n_0} K_{n_0}= 1,\]
which contradicts the assumption on $K_{n_0}$.
Consequently, for every $n\in \N^*$ we know that $K_n\neq 0$, thus $-a_n K_n\neq 0$.\\

We also know that 
\begin{equation}
    -a_n K_n= \lambda+ c_n \textrm{ with } \{c_n\}_{n\in \N^*}\in\ell^2(\N^{*}),
\end{equation}
which implies that at high frequency $\{|a_n K_n|\}_{n\geq M}$ is bounded away from 0 and also bounded from above. Therefore, $\{|a_n K_n|\}_{n\in \N^*}$ is bounded away from 0 and also bounded from above. Hence
\begin{equation}
    c<|K_n|<C, \quad \forall n\in \N^*. 
\end{equation}

By applying Lemma \ref{lem-rbheat-1} (4) with $m=0$ and $s\in (-3/2, 3/2)$, we get that $T:H^{s}_1\rightarrow H^{s}_1$ is an isomorphism \textcolor{black}{for any $s\in(-3/2,3/2)$}. The special case $s=0$ corresponds to the invertibility on the space $L^2_1$.\\

 \subsection{Conclusion}
 We are now in a position to prove Proposition \ref{thm-ope-t}.
 
 \begin{proof}
\textcolor{black}{The proof of Proposition \ref{thm-ope-t} is a consequence of the above subsections: we choose $K_{n}$ as in Sections \ref{sub:uniqueness}--\ref{subsec-Kn-bound}. We know from Section \ref{subsec-Kn-bound} that $K$ is a bounded functional on $H^{1/2+}_{1}$ and from Section \ref{sec-T-inv-L2}  we know that there exists $c>0$ and $C>0$ such that 
\begin{equation*}
    c<|K_{n}|<C,\quad \forall n\in \N^*. 
\end{equation*}
Consequently, from Lemma \ref{lem-rbheat-1} that $T$ is an isomorphism on $H^{s}_{1}$ for any $s\in(-3/2,3/2)$. Moreover from Lemma \ref{lem-op-eq-tbk} we know that for any $r\in(-1/2,1/2)$, and $\varphi\in H^{r+1}_1$ we have
\begin{equation*}
     (TA+ TBK)\varphi= (AT-\lambda T)\varphi\; \textrm{ in } H^{r-1}_1,
\end{equation*}
and from Lemma \ref{lem-tb=b}
\begin{equation}
    T\phi_{1}=\phi_{1}\;\text{ in }H^{-1/2-}_{1}.
\end{equation}
Finally from \eqref{eq:c_n} and the definition of $\{c_{n}\}_{n\in\mathbb{N}^{*}}$ given in \eqref{def-cn-kn}, for any $r\in[0,1/2)$ {\color{black} we have}
\begin{equation*}
\{(\lambda+a_{n}K_{n})n^{r}\}_{n\in\mathbb{N}^{*}}\in l^{2}(\N^{*}).    
\end{equation*}
Hence the operators $K$ and $T$ satisfy the properties announced in Proposition \ref{thm-ope-t} for $m=0$. Thanks to Lemma \ref{lem-rbheat-1} the same can be done identically in the case $m\neq0$. This ends the proof of Proposition \ref{thm-ope-t}.
}
 \end{proof}

 \section{Well-posedness results}
 \label{sec:well-pos}
 \subsection{Well-posedness of the  closed-loop systems of heat equations}
In this section we show that the closed-loop systems provided in Theorem \ref{thm-main-linear}   {\color{black}and Theorem \ref{thm-main-burgers} are} actually well-posed. 

{\color{black}We mainly focus on the well-posedness of the heat equation. In the next subsection a similar proof will yield the well-posedness of the viscous Burgers equation.}

It suffices to consider the case of $m=0$, while the other cases can be proved similarly.  Furthermore, for the ease of notations, here we only prove the following special case corresponding to $r=0$, while the other cases where $r\in (-1/2, 1/2)$ can be proved similarly. Inspired by the decomposition \eqref{eqk} and the fact that $H^{m}=H^{m}_{1}\oplus H^{m}_{2}$ it suffices to  consider the well-posedness  in $H^{m}_{1}$ and $H^{m}_{2}$ separately, which is given by the following lemma:
\begin{lem}[\textcolor{black}{Well-posedness of the odd and even part}: $m=0, r=0$]\label{lem-euofsol} 

{\color{black} Let $k\in \{1, 2\}$.}
 Let $y_0\in L^{2}_k$. Let $\phi\in H^{-1}_k$. Let $K_k: H^{3/4}_{k}\rightarrow \R$ be bounded. The equation
\begin{gather}\label{eq-phiky-l2}
    \begin{cases}
    \partial_t y-\Delta y= \phi_k K_k(y), \\
    y(0)= y_0,
    \end{cases}
\end{gather}
has a unique solution that is satisfied in $L^2_{loc}((0, +\infty); H^{-1}_k)$, \textcolor{black}{and}
\begin{equation}
      y(t)\in C^0([0, +\infty); L^2_k)\cap L^2_{loc}((0, +\infty); H^{1}_k)\cap H^1_{loc}((0, +\infty); H^{-1}_k).
  \end{equation}
\end{lem}
{\color{black}
\begin{cor}[Case $m=0$, $r\in (-1/2, 1/2)$]\label{cor-euofsol}

{\color{black} Let $k\in \{1, 2\}$.}
 Let $r\in (-1/2, 1/2)$. Let $y_0\in H^{r}_k$. Let $\phi\in H^{-1/2-}_k$.
 Let $K_k: H^{1/2+}_{k}\rightarrow \R$ be bounded. 
 The equation \eqref{eq-phiky-l2}
has a unique solution such that the equation is satisfied in $L^2_{loc}((0, +\infty); H^{r-1}_k)$, and
\begin{equation}\label{es-sol-op-r}
      y(t)\in C^0([0, +\infty); H^r_k)\cap L^2_{loc}((0, +\infty); H^{r+1}_k)\cap H^1_{loc}((0, +\infty); H^{r-1}_k).
  \end{equation}
\end{cor}
}

\begin{remark}[\textcolor{black}{Cases $m\neq 0, r\in (-1/2, 1/2)$}]\label{rmk-key} {\color{black} Let $k\in \{1, 2\}$.}
For the other cases where $m\neq 0$,  we are dealing with $y_0\in H^{m+r}_k, \phi_k\in H^{m-1/2-}_k, K_k: H^{m+1/2+}_k\rightarrow \R$. Either we can perform the same proof with respect to the pivot space $H^m_k$ but, \textcolor{black}{this time,}   or we can consider the isomorphism $D^m: H^m_k\rightarrow L^2_k$:
\begin{equation}
    D^m: n^{-m} f_n^{k} \mapsto f_n^{k},
\end{equation}
{\color{black} with convention that $D^m(f^2_0)= f^2_0$},
\textcolor{black}{where we recall that $f_{n}^{k}$ is an eigenfunction of $A$ given by \eqref{def:eigenvec}.} {\color{black}Observe that $D^m$ commute with Laplacian,}
thus the equation \eqref{eq-phiky-l2} is equivalent to 
\begin{gather}\label{eq-phiky-hm}
    \begin{cases}
    \partial_t w-\Delta w= (D^m \phi_k) K_kD^{-m}(w), \\
    w(0)= D^{-m} y_0\in H^r_k,
    \end{cases}
\end{gather}
with $y= D^m w$, which goes back to the case of Corollary \ref{cor-euofsol}.
\end{remark}

{\color{black}
By combining Lemma \ref{lem-euofsol}, Corollary \ref{cor-euofsol}, Remark \ref{rmk-key} and the fact that $H^m= H^m_1\oplus H^m_2$, we immediately get the well-posedness of the equation  \eqref{clo-lop-sys-lin}.
\begin{cor}[\textcolor{black}{Well-posedness of the heat equation \eqref{clo-lop-sys-lin}}]\label{cor-euofsol.key}
Let $m\in \R$. Let $r\in (-1/2, 1/2)$. Let $y_0\in H^{m+r}$. Let $\phi_k \in H^{m-1/2-}_k$ for every $k\in \{1, 2\}$.
 Let $K_k: H^{m+ 1/2+}\rightarrow \R$ be bounded satisfying $K_k: H^{m+ 1/2+}_{3-k}\rightarrow 0$ for every $k\in \{1, 2\}$. 
 The equation \eqref{clo-lop-sys-lin}
has a unique solution such that the equation holds in the space $L^2_{loc}((0, +\infty); H^{m+ r-1})$, and
\begin{equation}
      y(t)\in C^0([0, +\infty); H^{m+r})\cap L^2_{loc}((0, +\infty); H^{m+r+1})\cap H^1_{loc}((0, +\infty); H^{m+r-1}).
  \end{equation}
\end{cor}}

\subsection{Well-posedness of the viscous Burgers equation}
We now turn to the well-posedness of the closed-loop viscous Burgers system \eqref{clo-lop-sys-nonlin}. 
\begin{lem}[Well-posedness of the viscous Burgers equation \eqref{clo-lop-sys-nonlin}] 
\label{lem:wellposedburgers}
 Let $y_0\in L^{2}$. Let $\phi_1, \phi_2 \in H^{-1}$. Let $K_1, K_2: H^{3/4}\rightarrow \R$ be bounded. The equation
\begin{gather}\label{eq-phiky-l2-burg}
    \begin{cases}
    \partial_t y-\Delta y+ \partial_x (y^2/2)= \phi_1 K_1(y)+ \phi_2 K_2(y), \\
    y(0)= y_0,
    \end{cases}
\end{gather}
has a unique solution that is satisfied in $L^2_{loc}((0, +\infty); H^{-1})$ sense, \textcolor{black}{and}
\begin{equation}
      y(t)\in C^0([0, +\infty); L^2)\cap L^2_{loc}((0, +\infty); H^{1})\cap H^1_{loc}((0, +\infty); H^{-1}).
  \end{equation}
  \textcolor{black}{
  Moreover,
  \begin{equation}\label{eq:stabilitylarge}
      \|y(t,\cdot)\|_{L^{2}}\leq {\color{black}e^{Ct}\|y_{0}\|_{L^{2}},\;\forall\;t\in [0,+\infty).}
  \end{equation}
  }
\end{lem}

 \begin{proof}[Proof of Lemma \ref{lem-euofsol}]
By denoting $S(t)$ the free heat flow evolution:  $y:= S(t) y_0$ as the solution of
\begin{gather}
    \begin{cases}
    \partial_t y-\Delta y= 0\textrm{ in } \T, \\
    y(0)= y_0,
    \end{cases}
\end{gather}
we know from integration by parts that 
\begin{gather}
    \|S(t)y_0\|_{C^0([0, T]; L^2_1)}\leq \|y_0\|_{L^2_1}, \\
     \|S(t)y_0\|_{L^2([0, T]; H^1_1)}\leq \|y_0\|_{L^2_1}.
\end{gather}
We also know from integration by parts that  the solution $g$ of the inhomogeneous heat equation,
\begin{gather}
    \begin{cases}
    \partial_t g-\Delta g= f\textrm{ in } \T, \\
    g(0)= 0,
    \end{cases}
\end{gather}
satisfies, for any given $T>0$,
\begin{gather}
    \|g(t)\|_{C^0([0, T]; L^2_1)}\leq \|f\|_{L^2(0, T; H^{-1}_1)}, \\
     \|g(t)\|_{L^2([0, T]; H^1_1)}\leq \|f\|_{L^2(0, T; H^{-1}_1)}.
\end{gather}
Let us define
\begin{equation}
    \mathcal{B}_T:= \{y\in C^0([0, T]; L^2_1)\cap L^2([0, T]; H^1_1)\},
\end{equation}
with its norm given by
\begin{equation}
    \|y\|_{\mathcal{B}_T}:= \|y\|_{C^0([0, T]; L^2_1)}+ \|y\|_{L^2([0, T]; H^1_1)}.
\end{equation}
\textcolor{black}{
For $M>0$, we also define $\mathcal{B}_{T}(M)$ as 
\begin{equation}
        \mathcal{B}_T(M):= \{y\in \mathcal{B}_T\; |\; \|y\|_{\mathcal{B}_{T}}\leq M\}.
\end{equation}
}
Suppose that $\|y_0\|_{L^2_1}= R$. Then, we consider the map 
\begin{equation}
 \mathcal{L}:    z\in \mathcal{B}_T(3R)\mapsto y\in \mathcal{B}_T
\end{equation}
that is defined as 
\begin{gather}
    \begin{cases}
    \partial_t y-\Delta y= \phi K(z) \textrm{ in } \T, \\
    y(0)= y_0.
    \end{cases}
\end{gather}
We immediately know that $\textcolor{black}{y}\in \mathcal{B}_T$. However, in order to show that $\mathcal{L}$ is a contraction \textcolor{black}{on $\mathcal{B}_{T}(3R)$} we need more delicate estimates. Indeed, we benefit from the fact that $K$ is a functional on $H^{3/4}_1$ instead of on $H^1_1$. \\

We know that the solution $y= \mathcal{L} z$ satisfies
\begin{align}
    \|y\|_{\mathcal{B}_T}&\leq 2R+ 2\|\phi (Kz)\|_{L^2(0, T; H^{-1}_1)}, \notag \\
    &\leq 2R+ C \|K z\|_{L^2(0, T)}, \notag\\
    &\leq 2R+ C \|z\|_{L^2(0, T; H^{\frac{3}{4}}_1)},  \notag\\
    &\leq 2R+ C T^{\frac{1}{8}} \|z\|_{L^{\frac{8}{3}}(0, T; H^{\frac{3}{4}}_1)},  \notag\\
    &\leq    2R+ C T^{\frac{1}{8}} \|z\|_{\mathcal{B}_T},  \notag\\
    &\leq  2R+ 3CR T^{\frac{1}{8}},
\end{align}
where we have used the following \textcolor{black}{technical} lemma. 
\begin{lem}\label{lem:3/4}
\begin{equation}
    \|z\|_{L^{\frac{8}{3}}(0, T; H^{\frac{3}{4}}_1)}\leq \|z\|_{L^{\infty}(0, T; L^2_1)}^{\frac{1}{4}} \|z\|_{L^{2}(0, T; H^{1}_1)}^{\frac{3}{4}}\leq \|z\|_{\mathcal{B}_T}.
\end{equation}
\end{lem}
\begin{proof}
As we know from Sobolev interpolation that 
\begin{equation}
    \|f\|_{H^{3/4}_1}\leq  \|f\|_{L^2_1}^{\frac{1}{4}} \|f\|_{H^1_1}^{\frac{3}{4}},
\end{equation}
then further thanks Hölder inequality,
\begin{align*}
    \|z\|_{L^{\frac{8}{3}}(0, T; H^{\frac{3}{4}}_1)}&\leq  \| \|z\|_{L^2_1}^{\frac{1}{4}} \|z\|_{H^1_1}^{\frac{3}{4}} \|_{L^{\frac{8}{3}}(0, T)}, \\
    &\leq \| \|z\|_{L^2_1}^{\frac{1}{4}} \|_{L^{\infty}(0, T)}   \|  \|z\|_{H^1_1}^{\frac{3}{4}} \|_{L^{\frac{8}{3}}(0, T)}, \\
    &\leq \|z\|_{L^{\infty}(0, T; L^2_1)}^{\frac{1}{4}} \|z\|_{L^{2}(0, T; H^{1}_1)}^{\frac{3}{4}}, \\
    &\leq \|z\|_{\mathcal{B}_T}.
\end{align*}
\end{proof}
Therefore, for $T>0$ sufficiently small we know that 
\begin{equation}
    \mathcal{L}: \mathcal{B}_T(3R)\rightarrow \mathcal{B}_T(3R).
\end{equation}

Next, we show that, by choosing $T$ even small if necessary,  the map $\mathcal{L}$ is actually a contraction map.  For any $z_1, z_2\in \mathcal{B}_T(3R)$, suppose that $y_1:= \mathcal{L} z_1, y_2:= \mathcal{L} z_2$, thus
\begin{gather}
    \begin{cases}
    \partial_t y_1-\Delta y_1= \phi K(z_1) \textrm{ in } \T, \\
    y_1(0)= y_0,
    \end{cases}
\end{gather}
and 
\begin{gather}
    \begin{cases}
    \partial_t y_2-\Delta y_2= \phi K(z_2) \textrm{ in } \T, \\
    y_2(0)= y_0.
    \end{cases}
\end{gather}
This implies that $w:= y_1- y_2= \mathcal{L}(z_1-z_2)$ satisfies
\begin{gather}
    \begin{cases}
    \partial_t w-\Delta w= \phi K(z_1-z_2) \textrm{ in } \T, \\
    w(0)= 0,
    \end{cases}
\end{gather}
which further yields
\begin{equation}
    \|\mathcal{L}(z_1-z_2)\|_{\mathcal{B}_T}=  \|w\|_{\mathcal{B}_T}\leq 2\|\phi K(z_1- z_2)\|_{L^2(0, T; H^{-1}_1)} \leq CT^{\frac{1}{8}} \|z_1-z_2\|_{\mathcal{B}_T}.
\end{equation}

Therefore, $\mathcal{L}$ is actually a contraction on $\mathcal{B}_T(3R)$ for $T$ sufficiently small. 
Banach fixed point theorem gives the existence and uniqueness of the solution in a small time interval.\\

Finally, it is standard to extend the solution to a large time domain. It actually suffices to show the existence on $[0, 1]$, thus it does not blow up in this domain. Concerning the solution $y(t)$ of the system \eqref{eq-phiky-l2}, integration by parts, which together with Sobolev interpolation and Young's inequality, yield, 
\begin{align*}
    \frac{d}{dt}\|y(t)\|_{L^2_1}^2&\leq -2\|y(t)\|_{H^1_1}^2+ C \|y(t)\|_{H^1_1} |Ky(t)|,\\
    &\leq -2\|y(t)\|_{H^1_1}^2+ C \|y(t)\|_{H^1_1}\|y(t)\|_{H^{3/4}_1}, \\
    &\leq -2\|y(t)\|_{H^1_1}^2+ C \|y(t)\|_{H^1_1}^{\frac{7}{4}}\|y(t)\|_{L^2_1}^{\frac{1}{4}},\\
     &\leq -2\|y(t)\|_{H^1_1}^2+ C \left(\frac{7\varepsilon}{8}\|y(t)\|_{H^1_1}^{2}+ \frac{1}{8\varepsilon^7}\|y(t)\|_{L^2_1}^{2}\right), \\
       &\leq -\|y(t)\|_{H^1_1}^2+ C \|y(t)\|_{L^2_1}^{2}.
\end{align*}
The preceding \textit{a priori} estimate, to be combined with standard arguments, indicate the existence of solution on $[0, 1]$ and further on $[0, +\infty)$.
\end{proof}

{\color{black}
\begin{proof}[Proof of Corollary \ref{cor-euofsol}]
Concerning Corollary \ref{cor-euofsol}, again, we only prove the case that $k=1$ as the other case that $k=2$ is similar. Observe that $r-1<-1/2$, meaning $\phi_1\in H^{r-1}_1$.  \textcolor{black}{Meanwhile, the choice of $r$ also tells us that $K$ is bounded on $H^{r+1}_{1}$.}

At first we investigate the related  open-loop system, $i. e. $  we replace $K_1(y)$  by $u(t)\in L^2_{loc}(0, +\infty)$.
The equation
\begin{gather}
    \begin{cases}
    \partial_t y-\Delta y= \phi_1 u(t), \\
    y(0)= y_0\in H^r_1,
    \end{cases}
\end{gather}
has a unique solution in the space \eqref{es-sol-op-r}, satisfying
\begin{gather}
  \|y(t)\|_{C^0([0, T]; H^r_1)}^2\leq \|y_0\|_{H^r_1}^2+     \|\phi_1 u(t)\|_{L^2(0, T; H^{r-1}_1)}^2, \label{6238} \\
   \|y(t)\|_{L^2(0, T; H^{r+1}_1)}^2\leq \|y_0\|_{H^r_1}^2+  \|\phi_1 u(t)\|_{L^2(0, T; H^{r-1}_1)}^2. \label{6239}
\end{gather}
Indeed,
\begin{align*}
     \frac{1}{2} \frac{d}{dt} \|y(t)\|_{H^r_1}^2&= \langle y(t), \dot{y}(t)\rangle_{H^r_1} \\
     &= -\|y(t)\|_{H^{r+1}_1}^2 + \langle y(t), u(t)\phi_1 \rangle_{H^{r}_1}\\
        &\leq  -\|y(t)\|_{H^{r+1}_1}^2+ \|y(t)\|_{H^{r+1}_1} \|u(t)\phi_1\|_{H^{r-1}_1} \\
           &\leq  -\|y(t)\|_{H^{r+1}_1}^2+ \frac{1}{2}\|y(t)\|_{H^{r+1}_1}^1+ \frac{1}{2} \|u(t)\phi_1\|_{H^{r-1}_1}^2 \\
     &\leq  - \frac{1}{2}\|y(t)\|_{H^{r+1}_1}^2+ \frac{1}{2}|u(t)|^2 \|\phi_1\|_{H^{r-1}_1}^2,\\
      &\leq  - \frac{1}{2}\|y(t)\|_{H^{r+1}_1}^2+ C|u(t)|^2,
\end{align*}
where we have used the fact that for $h, g\in \mathcal{S}_1$ (thus extends to related Sobolev spaces)
\begin{gather*}
h:= \sum_{n\in \N^*} h_n \sin{nx},\;\;  g:= \sum_{n\in \N^*} g_n \sin{nx}, \\
    \langle h, g\rangle_{H^r_1}= \sum_{n\in \N^*} (n^r h_n) (n^r g_n)= \sum_{n\in \N^*} (n^{r+1} h_n) (n^{r-1} g_n)\leq ||h||_{H^{r+1}_1} ||g||_{H^{r-1}_1},
\end{gather*}
as well as that 
\begin{equation*}
    \langle h, \Delta h\rangle_{H^r_1}= -\sum_{n\in \N^*} (n^r h_n) (n^{r+2} h_n)= -\sum_{n\in \N^*} (n^{r+1} h_n) (n^{r+1} h_n)=- ||h||_{H^{r+1}_1}^2.
\end{equation*}

 Next, for the closed-loop system \eqref{eq-phiky-l2}, by the choice of  $r$ there exists some $s_0, s_1$ satisfying  $r< 1/2< s_0< r+1$  such that $H^{r+1}_1 \subset H^{s_0}_1\subset H^{r}_1$, and that $K$ is bounded on  $H^{s_0}_1$.  Consequently, the same proof of Lemma \ref{lem-euofsol} adapts here.  \textcolor{black}{For instance, } in Lemma \ref{lem-euofsol} the value of $s_0$ is chosen as $3/4$ \textcolor{black}{(see Lemma \ref{lem:3/4} in Section \ref{sec:well-pos})}. Indeed,
 \begin{lem} For $p:= \frac{2}{s_0-r}> 2$, we know that 
\begin{equation}
    \|z\|_{L^{p}(0, T; H^{s_0}_1)}\leq \|z\|_{L^{\infty}(0, T; H^r_1)}^{r+1-s_0} \|z\|_{L^{2}(0, T; H^{r+1}_1)}^{s_0-r}.
\end{equation}
\end{lem}
\begin{proof}
Since 
\begin{equation}
    \|z\|_{H^{s_0}_1}\leq  \|z\|_{H^{r}_1}^{r+1-s_0} \|z\|_{H^{r+1}_1}^{s_0-r},
\end{equation}
we know that 
\begin{align*}
    \|z\|_{L^{p}(0, T; H^{s_0}_1)}&\leq \| \|z\|_{ H^r_1}^{r+1-s_0} \|z\|_{ H^{r+1}_1}^{s_0-r}\|_{L^p(0, T)},\\
    &\leq \| \|z\|_{ H^r_1}^{r+1-s_0} \|_{L^{\infty}(0, T)}   \|  \|z\|_{ H^{r+1}_1}^{s_0-r}\|_{L^p(0, T)}, \\
    &= \|z\|_{ L^{\infty}(0, T; H^r_1)}^{r+1-s_0} \|z\|_{ L^{p(s_0- r)}(0, T; H^{r+1}_1)}^{s_0-r}.
\end{align*}
\end{proof}
For any given $T>0$. Let us define
\begin{equation}
    \mathcal{B}_T^r:= \{y\in C^0([0, T]; H^{r}_1)\cap L^2([0, T]; H^{r+1}_1)\},
\end{equation}
with its norm given by
\begin{equation}
    \|y\|_{\mathcal{B}^r_T}:= \|y\|_{C^0([0, T]; H^r_1)}+ \|y\|_{L^2([0, T]; H^{r+1}_1)},
\end{equation}
and 
\begin{equation}
        \mathcal{B}^r_T(M):= \{y\in \mathcal{B}^r_T\; |\; \|y\|_{\mathcal{B}^r_{T}}\leq M\}.
\end{equation}
Let $ \|y_0\|_{H^r_1}= R$. For any $T>0$, we further consider the map 
$\mathcal{L}^r: z\in \mathcal{B}^r_T(3R)\mapsto y\in \mathcal{B}^r_T$ defined as
\begin{gather}
    \begin{cases}
    \partial_t y-\Delta y= \phi K(z) \textrm{ in } \T, \\
    y(0)= y_0.
    \end{cases}
\end{gather}

  When  $T$ is sufficiently small, by adapting 
\begin{equation*}
    \|\phi (Kz)\|_{L^2(0, T; H^{r-1}_1)}\leq C \|Kz\|_{L^2(0, T)} \leq C \|z\|_{L^2(0, T; H^{s_0}_1)}\leq C T^{\frac{1}{2}-\frac{1}{p}} \|z\|_{L^{p}(0, T; H^{s_0}_1)},
\end{equation*}
and the fixed point argument, there is  a unique solution in the space \eqref{es-sol-op-r} in a small interval of time $[0, T]$, more precisely as the unique fixed point of $\mathcal{L}^r$ in $\mathcal{B}^r_T(3R)$.

Next, the \textit{a priori} estimate further implies the existence of a unique solution in large interval of time.
 \end{proof}
 \begin{remark}
 For the case that $k=2$ instead of 1, similar approach leads to the same well-posedness result. The only place that needs to be (slightly) modified is that the inequalites \eqref{6238}--\eqref{6239} should be replaced by, 
 \begin{gather*}
  \|y(t)\|_{C^0([0, T]; H^r_2)}^2\leq e^{2T}\left(\|y_0\|_{H^r_2}^2+     \|\phi_1 u(t)\|_{L^2(0, T; H^{r-1}_2)}^2\right),  \\
   \|y(t)\|_{L^2(0, T; H^{r+1}_2)}^2\leq  e^{2T}\left(\|y_0\|_{H^r_2}^2+  \|\phi_1 u(t)\|_{L^2(0, T; H^{r-1}_2)}^2\right).
\end{gather*}
That is because 
\begin{equation*}
  \langle y(t), \Delta y(t)  \rangle_{H^{r}_2, H^{r}_2}=-\|y(t)\|_{H^{r+1}_2}^2+ (\langle y, f^2_0\rangle)^2\leq -\|y(t)\|_{H^{r+1}_2}^2+ \|y(t)\|_{H^r_2}^2,
\end{equation*}
which leads to
\begin{align*}
     \frac{1}{2} \frac{d}{dt} \|y(t)\|_{H^r_2}^2
     &\leq -\|y(t)\|_{H^{r+1}_2}^2+ \|y(t)\|_{H^r_2}^2 + \langle y(t), u(t)\phi_1 \rangle_{H^{r}_2}\\
           &\leq  \|y(t)\|_{H^r_2}^2 -\frac{1}{2}\|y(t)\|_{H^{r+1}_2}^2+ \frac{1}{2} \|u(t)\phi_1\|_{H^{r-1}_2}^2.
\end{align*}
Thus
\begin{equation*}
    \|y(t)\|_{H^r_2}^2+ \|y\|_{L^2(0, t; H^{r+1}_2)}^2\leq e^{2t}\|y(0)\|_{H^r_2}^2+ e^{2t} \|u(t)\phi\|_{L^2(0, t; H^{r-1}_2)}^2. 
\end{equation*}
 \end{remark}

\begin{proof}[Proof of Lemma \ref{lem:wellposedburgers}]
Finally, we simply comment on the proof of Lemma \ref{cor-euofsol} whose proof is essentially the same as the proof of Lemma \ref{lem-euofsol}.

Indeed, concerning the existence of the solution in a small interval of time, it suffices to treat the nonlinear term as a perturbation  using Gagliardo–Nirenberg interpolation inequality,   which is standard.
\begin{align*}
    \|\partial_x(y^2)\|_{L^2(0, T; H^{-1})}^2&\leq  \|y y\|_{L^2(0, T; L^2)}^2, \\
    &\leq \int_0^T \|y(t, \cdot)\|_{L^2}^2 \|y(t, \cdot)\|_{L^{\infty}}^2 dt, \\
     &\leq C\int_0^T \|y(t, \cdot)\|_{L^2}^2 \|y(t, \cdot)\|_{H^{1/2}}^2 dt, \\
     &\leq C\int_0^T \|y(t, \cdot)\|_{L^2}^3 \|y(t, \cdot)\|_{H^1} dt, \\
     &\leq  C T^{\frac{1}{2}} \|y\|_{C^0([0, T]; L^2)}^3 \|y\|_{L^2(0, T; H^1)}.
\end{align*}
Next, classical energy estimates, benefiting from the fact that 
\begin{equation}
    \langle y, \partial_x(y^2) \rangle=0,
\end{equation}
lead to the existence of solution in large interval of time,
\begin{align*}
   \frac{1}{2} \frac{d}{dt}\|y(t)\|_{L^2}^2&= \langle y(t), \Delta y(t)+ \phi_1 K_1(y(t))+ \phi_2 K_2(y(t))  \rangle_{H^1, H^{-1}}\\
    &\leq -\|y(t)\|_{H^1}^2+ \|y(t)\|_{L^2}^2+ C \|y(t)\|_{H^1} \left(|K_1y(t)|+ |K_2y(t)|\right),\\
    &\leq -\frac{1}{2}\|y(t)\|_{H^1}^2+ C \|y(t)\|_{L^2}^{2},
\end{align*}
where, slightly different from the calculation on $L^2_1$ and $H^1_1$, for $H^1$ norm (just as for $H^1_2$ norm) we have
\begin{equation*}
  \langle y(t), \Delta y(t)  \rangle_{H^1, H^{-1}}=-\|y(t)\|_{H^1}^2+ (\langle y, f^2_0\rangle)^2\leq -\|y(t)\|_{H^1}^2+ \|y(t)\|_{L^2}^2. 
\end{equation*}
 \end{proof}

 }
 
 \section{Exponential stabilization of the heat and Burgers equations}
 \label{sec:proof-thms}
 \subsection{Heat equation: well-posedness and stability of the transformed system}
 \textcolor{black}{In order to show Theorem \ref{thm-main-linear}, we need to show the well-posedness of the closed-loop system and its exponential stability.  Also, to simplify the notation we assume that $m=0$ even though the exact same can be done with $m\neq 0$. Let $y_0\in L^2_1$. \textcolor{black}{Under the assumption of Theorem \ref{thm-main-linear} we can define $$B= (\phi_1, \phi_2), \; K=(K_{1},K_2)^T, \; T=T_{1}+T_2$$ given by Proposition \ref{thm-ope-t} and Corollary \ref{cor-key}.} }
 
 Let $\tau>0$.
  Considering the fact that for $k\in \{1, 2\}$ $\phi_{k}\in H^{-1}_k$ and that $K_{1}: H^1_k\rightarrow \R$,  for any $\varphi(t)\in L^2((0, \tau); H^1)$ we have that $B K\varphi(t)\in L^2((0, \tau); H^{-1})$. Then, concerning the closed-loop system
  \begin{equation}\label{sys-clo-yt}
      y_t-\Delta y= B K(y), \;\; y(0)=y_0,
  \end{equation}
  from Corollary \ref{cor-euofsol.key},  we get a unique solution of the closed-loop system
  \begin{equation}
      y(t)\in C^0([0, \tau]; L^2(\T^1))\cap L^2((0, \tau); H^1(\T^1))\cap H^1((0, \tau); H^{-1}(\T^1)),
  \end{equation}
  which indicates that the equation \eqref{sys-clo-yt} (thus each item of it) is satisfied in $L^2((0, \tau); H^{-1}(\T^1))$.

  Since the operator $T$ is bounded in $H^l(\T^1)$ space with $l= -1, 0, 1$ \textcolor{black}{(from Proposition \ref{thm-ope-t} and Corollary \ref{cor-key})}, we know that 
   \begin{equation}\label{reg-zt-H1}
      z(t):= Ty(t)\in C^0([0, \tau]; L^2(\T))\cap L^2((0, \tau); H^1(\T))\cap H^1((0, \tau); H^{-1}(\T)).
  \end{equation}
  
  Moreover, by \textcolor{black}{applying} $T$ to \eqref{sys-clo-yt} we know that 
  \begin{equation}\label{sys-clo-yt-T}
     \textcolor{black}{T y_t- TA y= TB K(y)} \;\text{ in }\;L^2((0, \tau); H^{-1}(\T^1)).
  \end{equation}
  By applying \textcolor{black}{Proposition \ref{thm-ope-t} and Corollary \ref{cor-key} with} $s=0$, together with the fact that $y(t)\in L^2((0, \tau); H^1_1(\T^1))$, we arrive at
   \begin{equation}\label{sys-clo-yt-T-trans}
     T y_t= ATy- \lambda T y,  \textrm{ in } L^2((0, \tau); H^{-1}(\T)).
  \end{equation}
Hence,
  \begin{equation}
      z_t-\Delta z+\lambda z=0 \textrm{ in } L^2((0, \tau); H^{-1}(\T)).
  \end{equation}
  Consequently, 
  \begin{equation}
     \frac{1}{2} \frac{d}{dt} \|z(t)\|_{L^2(\T)}^2= \langle z(t), \dot{z}(t) \rangle_{H^1, H^{-1}}\leq -\lambda
     \|z(t)\|_{L^2(\T)}^2
  \end{equation}
  holds in $L^1(0, T)$, which further implies the required decay property of the solution 
  \begin{equation}
      \|z(t)\|_{L^2}\leq e^{-\lambda t} \|z(0)\|_{L^2}, \; \forall t\in [0, \tau].
  \end{equation}
  
Finally,  
\textcolor{black}{as $T$ is an isomorphism on $L^{2}$ from Proposition \ref{thm-ope-t} and Corollary \ref{cor-key},} we conclude that
   \begin{equation}
      \|y(t)\|_{L^2}\leq C(\lambda) e^{-\lambda t} \|y^{0}\|_{L^2}, \; \forall t\in [0, \tau],
  \end{equation}
Given that this is true for any $\tau>0$ and that $C(\lambda)$ does not depend on $\tau$, it implies that
$y\in C^0([0, +\infty); L^2(\T))$ and 
\begin{equation}
\|y(t)\|_{L^2}\leq C(\lambda) e^{-\lambda t} \|y^{0}\|_{L^2}, \; \forall t\in [0, +\infty).
\end{equation}
This nearly ends the proof of Theorem \ref{thm-main-linear}, the only thing left is to check that we have also a stabilization in $H^{r}$ for any $r\in(-1/2,1/2)$.
 
\begin{remark}
Let $r\in (-1/2, 1/2)$. The same feedback law $K(y)$ also stabilizes the system \eqref{sys-clo-yt} in $H^{r}-$space.
\end{remark} 
Indeed, let $r\in (-1/2, 1/2)$ and let $y_0\in H^r$. As $\phi_k\in H^{r-1}_k$ and $K_k: H^{r+1}_k\rightarrow \R$, for any $\varphi(t)\in L^2((0, \tau); H^{r+1})$ we have that $B K\varphi(t)\in L^2((0, \tau); H^{r-1})$. Then equation \eqref{sys-clo-yt} has a solution 
\begin{equation}\label{reg-yt-Hs+1}
      y(t)\in C^0([0, \tau]; H^r(\T))\cap L^2((0, \tau); H^{r+1}(\T))\cap H^1((0, \tau); H^{r-1}(\T)),
  \end{equation}
 and holds in $L^2((0, \tau); H^{r-1}(\T))$.  
 
 Because $T$ is bounded in $H^{l}$ with $l= r-1, r, r+1$ (Lemma \ref{lem-rbheat-1} part (4) and Section \ref{subsec-Kn-bound}), we know that $z(t):= T y(t)$ lives in the  same space of $y(t)$ in \eqref{reg-yt-Hs+1}.
 
 By adapting $T$ to the equation \eqref{sys-clo-yt} we know that \eqref{sys-clo-yt-T} holds in $L^2((0, \tau); H^{r-1}(\T))$.  
 Then by adapting Lemma \ref{lem-op-eq-tbk} to the case that $s=r$, we get 
 \begin{equation}
     T y_t= ATy- \lambda Ty,  \textrm{ in } L^2((0, \tau); H^{r-1}(\T)).
  \end{equation}
Hence,
  \begin{equation}\label{decay-z}
      z_t-\Delta z+\lambda z=0 \textrm{ in } L^2((0, \tau); H^{r-1}(\T)),
  \end{equation}
which leads to the required exponential decay \textcolor{black}{of $z$} in $H^r$,
\begin{equation}
     \frac{1}{2} \frac{d}{dt} \|z(t)\|_{H^r}^2= \langle z(t), \dot{z}(t) \rangle_{H^{r}}\leq  -\lambda \|z(t)\|_{H^r}^2. 
\end{equation}

\textcolor{black}{Consequently, using again that $T$ is an isomorphism in $H^{r}$,}
\begin{equation}
       \|y(t)\|_{H^{r}}\leq \textcolor{black}{C_r(\lambda) e^{-\lambda t}}\|y(0)\|_{H^{r}},
\end{equation}
with \textcolor{black}{$C$ and $C_r(\lambda)$} depending on $r\in (-1/2, 1/2)$ and $\lambda\notin \mathcal{N}$. \textcolor{black}{This ends the proof of Theorem \ref{thm-main-linear}}.

{\color{black}
\subsection{Viscous Burgers equation: well-posedness of the target system and stability of the closed-loop system}
\textcolor{black}{The proof of Theorem \ref{thm-main-burgers} dealing with the viscous Burgers equation is very similar to the proof of Theorem \ref{thm-main-linear} dealing with the} heat equation with $m=0, r=0$.
 Let \textcolor{black}{$\tau>0$ and} $y_0\in L^2_1$ \textcolor{black}{such that $\|y_0\|_{L^{2}}<\delta$, where $\delta$ is a constant to be chosen.}  \textcolor{black}{Lemma \ref{lem:wellposedburgers}} implies that the closed-loop system \eqref{clo-lop-sys-nonlin} has a unique solution $y(t)$, \textcolor{black}{provided that $\delta$ is sufficiently small (depending on $\tau$),} and
  \begin{equation}\label{reg-yt-H1}
      y(t)\in C^0([0, \tau]; L^2(\T))\cap L^2((0, \tau); H^1(\T))\cap H^1((0, \tau); H^{-1}(\T)),
  \end{equation}
  which holds in $L^2(0, \tau; H^{-1}(\T))$. Next, again, thanks to the fact that 
   the operator $T$ is bounded in $H^l_1(\T)$ space with $l= -1, 0, 1$,
   \begin{equation}
           z(t):= Ty(t)\in C^0([0, \tau]; L^2(\T))\cap L^2((0, \tau); H^1(\T))\cap H^1((0, \tau); H^{-1}(\T)).
  \end{equation}
  
  Next, by \textcolor{black}{applying} $T$ to \eqref{clo-lop-sys-nonlin} we know that in $L^2((0, \tau); H^{-1}(\T))$,
  \begin{equation}
     T y_t- TA y+ T\partial_x(y^2/2)= T B K(y), 
  \end{equation}
  where we used the fact that $T\partial_x(y^2/2)\in L^2((0, \tau); H^{-1})$.
  
  By applying Lemma \ref{lem-op-eq-tbk} for the case $s=0$, combined with the fact that $y(t)\in L^2((0, \tau); H^1(\T^1))$, we get that
   \begin{equation}
   T y_t= ATy- \lambda Ty- T\partial_x(y^2)/2,  \textrm{ in } L^2((0, \tau); H^{-1}(\T)).
  \end{equation}
Hence,
  \begin{equation}
      z_t-\Delta z+\lambda z+ T\partial_x(T^{-1}z)^2/2 =0 \textrm{ in } L^2((0, \tau); H^{-1}_1(\T)).
  \end{equation}
  
  Therefore, the system is locally stable in $L^2(\T)$ space, provided that $\|z\|_{L^{2}}$ is small enough, indeed
  \begin{align*}
       \frac{1}{2} \frac{d}{dt} \|z(t)\|_{L^2(\T)}^2&= \langle z(t), \dot{z}(t) \rangle_{H^1, H^{-1}} \\
       &= \langle z(t), \Delta z-\lambda z- T\partial_x(T^{-1}z)^2/2 \rangle_{H^1, H^{-1}}, \\
       &\leq -\|z\|_{H^1}^2+ \|z\|_{L^2}^2-\lambda \|z\|_{L^2}^2+ C \|z\|_{H^1} \|T \partial_x(T^{-1}z)^2\|_{H^{-1}}\\
        &\leq -\|z\|_{H^1}^2-(\lambda-1) \|z\|_{L^2}^2+ C \|z\|_{H^1} \|(T^{-1}z)^2\|_{L^2} \\
          &\leq -\|z\|_{H^1}^2-(\lambda-1) \|z\|_{L^2}^2+ C \|z\|_{H^1} \|T^{-1}z\|_{L^2}\|T^{-1}z\|_{H^1} \\
         &\leq -\|z\|_{H^1}^2-(\lambda-1) \|z\|_{L^2}^2+ C \|z\|_{H^1}^2 \|z\|_{L^2}.
  \end{align*}
\textcolor{black}{Hence, provided that $\sup_{[0,\tau]}(\|z\|_{L^{2}})$ is small enough (depending on $\lambda\in (1, +\infty)$) one has
  \begin{equation}
        \|z(t)\|_{L^2}\leq  e^{-(\lambda-1) t}  \|z(0)\|_{L^2}, \forall t\in[0, \tau],
  \end{equation}}
    which also implies, using $T^{-1}$, that 
  \begin{equation}\label{eq:stab-burgers}
      \|y(t)\|_{L^2}\leq C (\lambda) e^{-(\lambda-1) t}  \|y(0)\|_{L^2}, \forall t\in[0, \tau],
  \end{equation}
\textcolor{black}{provided that $\sup_{[0,\tau]}(\|z\|_{L^{2}})$ small, or equivalently that $\sup_{[0,\tau]}(\|y\|_{L^{2}})$ small, from the isomorphism property of $T$. Finally, from Lemma \ref{lem:wellposedburgers} and \eqref{eq:stabilitylarge} it suffices to have $\|y_{0}\|_{L^{2}}$ small. This means that there exists $\delta_{1}(\tau,\lambda)$ such that for any $\delta\in(0,\delta_{1}(\tau,\lambda))$ the solution $y$ satisfies \eqref{reg-yt-H1} and the exponential stability estimate \eqref{eq:stab-burgers} holds.}\\

\textcolor{black}{So far the constant $\delta$ depends on $\tau$ but, leveraging the exponential stability estimate \eqref{eq:stab-burgers}, it can be made independent of $\tau$ using a very classical argument:
let $\tau_{1}>0$ such that $e^{-(\lambda-1) \tau_{1}/2}< (C(\lambda))^{{\color{black}-1}}$ , and select $\delta = \delta_{1}(\tau_{1},\lambda)$, then $y$ exists and \eqref{eq:stab-burgers} holds on $[0,\tau_{1}]$, therefore
\begin{equation}\label{eq:yt1}
   \|y(\tau_{1},\cdot)\|_{L^{2}}\leq e^{-(\lambda-1)\tau_{1}/2}\|y_{0}\|_{L^{2}}.
\end{equation}
As the system \eqref{clo-lop-sys-nonlin} is autonomous, studying it on $[\tau_{1},2\tau_{1}]$ is the same as studying it on $[0,\tau_{1}]$ with initial condition $y(\tau_{1},\cdot)$. And from \eqref{eq:yt1}, $\|y(\tau_{1},\cdot)\|_{L^{2}}{\color{black}\leq \|y_0\|_{L^2}}\leq \delta_{1}(\tau_{1},\lambda)$, hence the solution exists on $[\tau_{1},2\tau_{1}]$ and 
\begin{equation}
          \|y(t)\|_{L^2}\leq C (\lambda) e^{-(\lambda-1) (t-\tau_{1})}  \|y(\tau_{1})\|_{L^2}, \forall t\in[\tau_{1}, 2\tau_{1}],
\end{equation}
which together with \eqref{eq:stab-burgers} gives
\begin{equation}
          \|y(t)\|_{L^2}\leq C (\lambda) e^{-\frac{(\lambda-1)}{2} t}            \|y(0)\|_{L^2}, \forall t\in[0, 2\tau_{1}],
\end{equation}
Hence, iterating this procedure, for any $n\in\mathbb{N}^{*}$ $y$ exists on $[0,n\tau_{1}]$
\begin{equation}
          \|y(t)\|_{L^2}\leq C (\lambda) e^{-\frac{(\lambda-1)}{2} t}            \|y(0)\|_{L^2}, \forall t\in[0, n\tau_{1}],
\end{equation}
hence $y$ exists on $[0,+\infty)$ and
\begin{equation}
          \|y(t)\|_{L^2}\leq C (\lambda) e^{-\frac{(\lambda-1)}{2} t}            \|y(0)\|_{L^2}, \forall t\in[0, +\infty).
\end{equation}
This ends the proof of Theorem \ref{thm-main-burgers}.}
}

\section{Conclusion}

\subsection{Quantitative studies on $C_{r}(\lambda, m)$}\label{subsection:quantitative}
Thanks to the precise analysis introduced in this paper. The next step could also be on the quantitative study of the stabilization cost, namely on the value of the constant $C_r(\lambda, m)$ in Theorem \ref{thm-main-linear}. For example, let $m=0$. 
For any fixed $\lambda\notin \mathcal{N}$, even if we do not have enough information on the exact value of $C_r(\lambda)$, it can be conjectured that the optimal value of $C_r(\lambda)$ (at least obeying our feedback law) tends to $+\infty$ as $|r|$ tends to $1/2^-$.

The appearance of the critical set $\mathcal{N}$ also indicates that, for any $r\in (-1/2, 1/2)$ fixed, as $\lambda$ tends to $\mathcal{N}$ the value of  $C_r(\lambda)$ tends to $+\infty$.  Therefore, it seems that by adapting this feedback we are not able to achieve $e^{C\sqrt{\lambda}}$ type estimates, at least not uniformly on $\lambda\in \R^+$.  
However, we believe that with the precise functional settings treated in this paper, we are much more closed to such quantitative results.  Indeed, it is still possible and reasonable to expect $e^{C\sqrt{\lambda}}$ estimate on $\{\lambda= 4N+2; N\in \N^*\}$.
\subsection{General parabolic equations}
It is natural to ask whether our framework also adapts  general parabolic equations, namely 
\begin{gather}
    \begin{cases}
    \partial_t y-\Delta y+ a_1(x) \partial_x y+ a_2(x) y= \phi_{1} K_{1}(y)+\phi_{2} K_{2}(y), \\    y(0)= y_0,
    \end{cases}
\end{gather}
with $a_i(x)$ satisfying suitable regularity assumption. 

By regarding the lower order operators $a_1(x) \partial_x + a_2(x)$ as source, the same feedback law and transformation yields the operator equality. However, on the next step, when applying $T$ to the evolution equation, the source   $a_1(x) \partial_x y + a_2(x)y$ turned out to be $T\left(a_1(x) \partial_x (T^{-1}z)+ a_2(x) (T^{-1}z)\right)$ which may become even stronger than the $-\lambda z$ damping produced by backstepping. 

Therefore, it seems that we need to perform backstepping directly on the elliptic operator $-\Delta +a_1(x) \partial_x + a_2(x)$. In the case that $a_1(x)=0$, the analysis is probably simpler as the operator is remained to be self-adjoint. However, losing those explicit formulation of eigenvalues and eigenfunctions make it more challenging to conclude Lemma \ref{lem-rbheat-1}. While the other case that $a_1(x)\neq 0$ is of course more delicate,  maybe the perturbation theory of  resolvent estimates should be applied, a good news is that due to the spectral gaps between different eigenvalues are increasing, it is possible that no smallness of $a_1(x) \partial_x + a_2(x)$ should be assumed.  Technically speaking, due to the appearance of  eigenvalues admitting double multiplicity the bifurcation phenomenon when splitting those eigenvalues should appear,  the resolvent analysis involved  would be more interesting and more delicate to some related works such as \cite{Beauchard05, coron:hal-03161523}.

\subsection{Stabilization with one scalar control}
It is proved in this work that two scalar controls are necessary and sufficient for the rapid stabilization of the heat equation provided some decay information, because of those double eigenvalues.   But if we work on more general parabolic equations, for which it is possible that every eigenvalues are simple and isolated, then probably one scalar control, of course always admitting suitable decay properties, is sufficient to conclude controllability and rapid stabilization.   

According to the ``return philosophy" introduced by Coron \cite{CoronBook}, it is still possible to stabilize nonlinear system even if the linearized system is not stabilizable. Therefore, it is also of interest to consider the rapid stabilization  of the viscous Burgers equation with one scalar control.

\section*{Acknowledgements}
Ludovick Gagnon was partially supported by the French Grant ANR ODISSE (ANR-19-CE48-0004-01) and the French Grant ANR TRECOS (ANR-20-CE40-0009). Amaury Hayat was financially supported by Ecole des Ponts Paristech.  Shengquan Xiang was financially supported by the Chair of Partial Differential Equations at EPFL. Christophe Zhang was partially funded by the Chair Dynamics Control and Numerics (Alexander von Humboldt Professorship) of the Department of Data Science of the Friedrich Alexander Universität Erlangen-Nürnberg, and by the INRIA Grand-Est.

\bibliographystyle{plain}    
\bibliography{biblio}

\end{document}